\documentclass[11pt]{amsart}
\usepackage{amsmath,amsfonts,amssymb,amsthm}
\usepackage{graphicx}
\usepackage{color}
\usepackage{hyperref}
\usepackage{enumerate}

\def\z{\mathfrak{z}}
\def\b{\mathfrak{b}}
\def\u{\mathfrak{u}}
\def\f{\mathfrak{f}}
\def\k{\mathfrak{k}}
\def\g{\mathfrak{g}}
\def\h{\mathfrak{h}}
\def\l{\mathfrak{l}}
\def\m{\mathfrak{m}}
\def\n{\mathfrak{n}}
\def\v{\mathfrak{v}}
\def\a{\mathfrak{a}}
\def\o{\mathfrak{o}}
\def\p{\mathfrak{p}}
\def\q{\mathfrak{q}}
\def\s{\mathfrak{s}}

\def\hcx{\{J_{\alpha}\}}

\def\C{\mathbb{C}}
\def\F{\mathbb{F}}
\def\R{\mathbb{R}}

\def\N{\mathbb{N}}
\def\H{\mathbb H}
\def\J{\mathbb{J}}

\def\al{\alpha}

\def\ad{\operatorname{ad}{\:\!\!}}

\def\alt{\raise1pt\hbox{$\bigwedge$}}
\def\pint{\langle \cdotp,\cdotp \rangle }

\def\la{\langle}
\def\ra{\rangle}

\theoremstyle{plain}
\newtheorem{theorem}{\bf Theorem}[section]
\newtheorem{corollary}[theorem]{\bf Corollary}
\newtheorem{proposition}[theorem]{\bf Proposition}
\newtheorem{lemma}[theorem]{\bf Lemma}

\theoremstyle{definition}
\newtheorem{definition}[theorem]{\bf Definition}
\newtheorem{example}[theorem]{\bf Example}

\theoremstyle{remark}
\newtheorem{remark}[theorem]{\bf Remark}

\newcommand{\ri}{{\rm (i)}}
\newcommand{\rii}{{\rm (ii)}}
\newcommand{\riii}{{\rm (iii)}}
\newcommand{\riv}{{\rm (iv)}}
\newcommand{\rv}{{\rm (v)}}
\newcommand{\rlp}{{\rm (}}
\newcommand{\rrp}{{\rm )}}

\newcommand{\red}{\textcolor{red}}

\title{Complex structures on two-step nilpotent Lie groups}

\author{Mar\'ia Laura Barberis}
\email{mlbarberis@unc.edu.ar}

\date{January 2023}
\address{FAMAF, Universidad Nacional de C\'ordoba and CIEM-CONICET, Ciudad Universitaria, Av. Medina Allende s/n, X5000HUA C\'ordoba, Argentina}

\subjclass[2020]{17B30, 22E25,  53C15, 53C30}
\keywords{Complex structure, two-step nilpotent  Lie group}

\begin{document}

\begin{abstract} We give a characterization of the $2$-step nilpotent Lie algebras whose 
corresponding Lie groups admit a left invariant complex structure. This is done by considering separately the cases when the complex structure is 2-step or 3-step nilpotent in the sense of \cite{CFGU}.

We also study the Hermitian geometry 2-step nilpotent Lie groups.  
We show that if a left invariant Hermitian metric on such  Lie group is pluriclosed, then  the corresponding complex structure is 2-step nilpotent.  Moreover, we obtain a necessary and sufficient condition for such a metric to be pluriclosed in case the complex structure is abelian. 
	This allows us to   
	 show that pluriclosed metrics on nilpotent Lie algebras with one dimensional commutator ideal can only occur on trivial central extensions of the $3$-dimensional Heisenberg Lie algebra.   
	
    We show   that certain  Hermitian nilmanifolds constructed from compact semisimple irreducible Hermitian symmetric pairs are pluriclosed  with respect to a canonically defined abelian complex structure.

Finally, we give necessary and sufficient conditions for a naturally reductive Riemannian metric on a  nilmanifold to be Hermitian with respect to an abelian complex structure. We prove the analogue of this result in the hypercomplex case, 
	thereby obtaining  a distinguished  family of hyper-K\"ahler with torsion metrics.

\end{abstract}

\maketitle

\section{Introduction}
Two-step nilpotent Lie groups with left invariant geometric structures play an important role in differential geometry. They have provided examples and counterexamples to open questions and 
conjectures and have been extensively studied. They are the simplest version of Carnot groups, that is, simply connected Lie groups whose Lie algebra is stratified, a notion introduced by G. Folland  \cite{Fol}  which yielded a source of examples of interest in analysis \cite{Good},   optimal control theory \cite{Bro} and quasiconformal theory \cite{Pansu}. It was proved in \cite{G} that naturally reductive homogeneous nilmanifolds are at most $2$-step, and these were described in \cite{Lau} via representations of compact Lie groups. 
We refer to \cite{Eb1,Eb2} for general theory  of left invariant   Riemannian metrics on $2$-step nilpotent Lie groups. The study of    symplectic structures on $2$-step nilpotent Lie groups was carried out in  \cite{OS}, and  a general theory in the nilpotent case was developed in \cite{Gua}. 

The motivation for studying the structure of $2$-step nilpotent Lie groups endowed with a left invariant complex structure comes from several applications in differential geometry. For instance, 
it was proved in \cite{DF1} that $8$-dimensional nilpotent Lie groups carrying left invariant hypercomplex structures are necessarily $2$-step nilpotent and 
it has been shown in \cite{AN} that if a complex nilmanifold admits an invariant pluriclosed (or SKT) metric then it is a torus or  the corresponding  Lie group is $2$-step nilpotent. 

An important feature of complex nilmanifolds is that they have holomorphically trivial canonical bundle 
\cite{BDV,CG}. 
The classification of nilpotent Lie groups endowed with a left invariant complex structure is a difficult problem. There are some results in low dimensions \cite{COUV,LUV,LUV1,Sal,Ug}, but there is no general theory even in the case when the group is $2$-step nilpotent. In this case, 
some results were obtained in  \cite{De,DLV,GZZ,Z} and  when the complex structure is abelian \cite{Bar} or 
bi-invariant \cite{AD,BM,De}. 
There is a notion of nilpotent complex structure, introduced in \cite{CFGU} (see \S\ref{nilpot-cx-str} below). It was shown in \cite{GZZ} (see also \cite{Z}) that a complex structure on a $2$-step nilpotent Lie algebra is nilpotent of step $2$ or $3$. In this work, we will start by describing  the nilpotent complex structures of step $2$ and we will show that a nilpotent complex structure of step $3$ can be obtained from one of step $2$ and from an abelian complex structure.

\section{Preliminaries}
Let $G$ be a connected real Lie group with Lie algebra 
 $\g$.  A complex structure on $G$ is an automorphism
$J$ of the tangent bundle $TG$ satisfying $J^2=-I$ and the 
integrability condition $N_{J}(X,Y) =0$ for all $X,Y$ vector fields on $G,$ where $N_J$ is the Nijenhuis tensor:
\begin{equation}  N_{J}(X,Y) =
[X,Y]+J([JX,Y]+[X,JY])-[JX,JY].  \label{nijen} \end{equation}
  If
$J$ is left invariant, that is, if left translations by elements of $G$ are holomorphic maps, 
 then $J$ is determined 
by the value at the identity of $G$. Thus,
a left invariant complex
structure on $G$ 
amounts to a  complex structure 
on its Lie algebra $\g$, that is,
a real linear endomorphism 
$J$ of $\g$ satisfying $J^2 = -I$ and $N_J(x, y)=0$ for all $x, y$ in $\g.$ 

A complex structure $J$ on a Lie algebra $\g$ is called \textit{abelian} when $[Jx,Jy]=[x,y]$ for all $x,y\in\g$, and $J$ is called \textit{bi-invariant} when $J[x,y]=[x,Jy]$ for all $x,y\in\g$. 

A  Hermitian structure on $\g$ is a pair $(J,\pint )$ of a complex structure $J$ and a positive definite symmetric bilinear form $\pint$ compatible with $J$, that is, $\la Jx, Jy\ra =\la x, y\ra $ for all $x, y$ in $\g.$ 

Given a Hermitian structure $(J,\pint )$ on $\g$, there exists a unique $\g$-valued bilinear form $\nabla ^B$ on $\g$, $\nabla^B :\g \times \g \to \g$, $(x,y)\mapsto \nabla^B _x y$, satisfying:
\begin{equation}\label{eq:bismut}
  \nabla^B _x \text{ is skew-symmetric  $\forall \, x$}, \quad \nabla^BJ=0, \quad c(x,y,z):= \la  x, T^B(y,z)\ra \text{ is a 3-form},
\end{equation}
where $T^B(y,z)=\nabla^B_y z- \nabla^B_z y -[y,z]$ is the torsion of $\nabla^B$. If $G$ is a Lie group with Lie algebra $\g $ and $(J,g)$ is the  left invariant Hermitian structure on $G$ induced by  $(J,\pint )$, then the left invariant  affine connection determined by $\nabla^B$ is called the Bismut or Strominger connection of $(J, g)$ (see \cite{Bis,Str}). 

The torsion 3-form $c$ defined in \eqref{eq:bismut} 
 can be computed by:
\[ c(x, y, z) = -\la [Jx, Jy], Z\ra - \la [Jy, J z],x\ra - \la [J z, Jx], y \ra,  \quad x, y , z\in  \g , \]
 (see \cite[Equation (3.2)]{EFV}). The inner product $\pint$ is called \emph{pluriclosed} (or SKT) with respect to $J$ if $c$ is closed. 
The exterior derivative of $c$ is  given by:
\begin{align}   dc(w,u, y , z) = &\la [J [w,u], Jy], z\ra  + \la [Jy, J z], [w,u]\ra  + \la [J z, J [w,u]], y \ra \nonumber \\
& - \la [J [w,  y ], Ju], z\ra  - \la [Ju, J z], [w, y ]\ra  - \la [J z, J [w, y ]],u\ra \nonumber \\
& + \la[J [w, z], Ju], y\ra  + \la [Ju, Jy], [w, z]\ra  + \la [Jy, J [w, z]],u\ra \label{eq:dc}\\
& + \la [J [u, y ], Jw], z\ra  + \la [Jw, J z], [u, y ]\ra  + \la [J z, J [u, y ]],w\ra \nonumber  \\
& - \la [J [u, z], Jw], y\ra  - \la [Jw, Jy], [u, z]\ra  - \la [Jy, J [u, z]],w\ra \nonumber  \\
& + \la [J [y , z], Jw],u\ra  + \la [Jw, Ju], [y , z]\ra  + \la [Ju, J [y , z]],w\ra \nonumber  ,
\end{align}
for $w,u,y,z\in \g$.
\section{Nilpotent complex structures}\label{nilpot-cx-str}
Given a Lie algebra $\g$, the ascending central series of $\g$ is defined inductively by
\begin{equation} \g_0= 0 , \qquad \g_\ell=\{ x\in \g :[x , \n ]\subset \g_{\ell-1} \},\;\; \ell\geq 1 .
\end{equation}
Each $\g_\ell$ is an ideal of $\g$ and  $\g$ is called $s$-step  nilpotent,  $s\geq 1$, if $\g_s =\g$ and $\g_{s-1}\neq \g$. Note that $\g _1=\z,$ the center of $\g$, and $\g$ is $2$-step nilpotent if and only if  $\g$ is non-abelian and $\g '\subset\z$, where  $\g '=[\g , \g]$ is the commutator ideal of $\g$. 

Let $J$ be a complex structure on a  Lie algebra  $\g $ and define inductively  the following ascending series of $\g$ (see \cite{CFGU}):
\begin{equation}\label{eq:a_l} \a_0(J)=0, \qquad \a_{\ell}(J)=\left\{x\in \g : [x,\g ]\subset \a_{\ell-1}(J) \text{ and } [Jx,\g ]\subset \a_{\ell-1}(J) \right\}, \;\; \ell\geq 1.
\end{equation}
Note that $\a_{1}(J)=\z\cap J\z$ and 
each $\a_{\ell}(J)$ is a $J$-invariant ideal of $\g$. 
The complex structure $J$ on $\g $ is called {\em nilpotent} \cite{CFGU} if $\a_{t}(J)=\n $ for some positive integer $t$. We will say that a nilpotent complex structure  $J$ is {\em $t$-step } \cite{Z} if $t$ is the smallest such $s$. 
In \cite{GZZ} the positive integer $t$, which depends on $J$, was denoted by $\nu(J)$. 

It follows from the fact that $\a_{\ell}(J)\subset \g_\ell$ for all $\ell$, that if $J$ is a nilpotent complex structure on $\g$ then $\g$ is necessarily a nilpotent Lie algebra. On the other hand, there are nilpotent Lie algebras $\g$ such that no complex structure on $\g$ is nilpotent. This is the case when the Lie algebra has one dimensional center, as the next result shows.  
\begin{lemma}\label{lem:non-nilp-J} Let $\g$ be a nilpotent Lie algebra such that $\dim \z =1$. Then no complex structure on $\g$ is nilpotent.
    \end{lemma}
\begin{proof}
Let $J$ be a complex structure on $\g$. The ideal $\a_{1}(J)=\z\cap J\z $ is zero or even dimensional. Since $\z\cap J\z \subset \z$ and  $\dim \z =1$, it followss that $\z\cap J\z =0$, that is, $\a_{1}(J)=0$. This  implies  $\a_{\ell}(J)=0$ for all $\ell$, so $J$ is not nilpotent. 
\end{proof}

As a consequence of the above lemma, 
the Lie algebras considered in \cite{LUV1} do not admit nilpotent complex structures. 
Another example of a nilpotent Lie algebra such that all its complex structures are non-nilpotent can be found in \cite[Section 4.2]{Rol}. 
In contrast with this situation,  when $\n$ is a $2$-step nilpotent Lie algebra, any complex structure on $\n$ is  $t$-step with  $t=2$ or $3$ (see Theorem \ref{2or3} below). 
It should be pointed out that if $J$ is  $t$-step  then the sequence $\left\{\a_{\ell}(J)\right\}_{\ell\geq 0}$  increases strictly till $\ell=t$ (see \cite{CFGU} for details), and not every complex structure on $\n$ is nilpotent. Moreover, 
 examples of nilpotent Lie algebras carrying both  nilpotent and non-nilpotent 
complex structures can be found in \cite{CFP,Rol}. 

In contrast with Lemma \ref{lem:non-nilp-J}, 
it was proved in \cite[Proposition 3.3]{Rol} that when the Lie algebra $\n$ is $2$-step nilpotent, any complex structure $J$ on $\n$  is necessarily nilpotent, and the following   result states that, moreover, $J$ is $t$-step for $t=2$ or $3$ (see also \cite[Theorem 3.9]{Z}). 

\begin{theorem}[{\cite[Theorem 3.1]{GZZ}}] \label{2or3} If $\n$ is a $2$-step nilpotent Lie algebra and $J$ is a complex structure on $\n$, then $J$ is $t$-step for $t=2$ or $3$.  
\end{theorem}

If $\n$ is $s$-step nilpotent, $\dim \n=2n$,  and $J$ is $t$-step,   then $s\leq t\leq n$ (see \cite{CFGU}). In particular, if $J$ is $2$-step  then  $\n$ is necessarily $2$-step nilpotent. This fact also follows from the next lemma, which gives  equivalent conditions for $J$ to be 2-step. 
We denote by $\n '$ the commutator ideal $\n '=[\n , \n ] $. 
\begin{lemma}\label{J-2-pasos} Let $\n$ be a  nilpotent Lie algebra with a complex structure $J$. The following conditions are equivalent:
\begin{enumerate}
\item[$\ri$] $J$  is 2-step,
\item[$\rii$]  $\n$ is 2-step nilpotent and $J\n '\subset \z$, 
\item[$\riii$] $\n '\subset \z\cap J\z$, 
\item[$\riv$] there exists a $J$-invariant ideal $\z_0$ of $\n$ such that $\n' \subset \z_0 \subset \z$.
\end{enumerate}
\end{lemma}
\begin{proof} Note that $J$ is 2-step if and only if $\a_{2}(J)=\n$, and this amounts to  $\n '\subset \z\cap J\z$. Therefore, (i) and (iii) are equivalent. 

Assume that (iii) holds. Then  $\n$ is 2-step nilpotent and  $J\n ' \subset \z\cap J\z \subset \z $, so   (ii) follows. Conversely, assume that (ii) is satisfied, then $\n '\subset J\z$ and also $\n '\subset \z$ since $\n$ is $2$-step nilpotent. Therefore,  $\n '\subset  \z\cap J\z$, that is, (iii) holds.



If (ii) holds, then $J\n '\subset \z$ and also $\n'\subset \z$ since $\n$ is $2$-step nilpotent. Therefore,  (iv) is satisfied for $\z_0= \n ' +J \n '$. On the other hand, 
if  (iv) is satisfied, then $J\n'\subset J\z_0\subset \z_0\subset \z$. Moreover, $\n'\subset \z _0 \subset \z$, hence, $\n'\subset  \z$, that is, $\n$ is 2-step nilpotent. Therefore, (ii) holds, and the lemma follows.  
\end{proof}

Given a 2-step complex structure, the ideal $\z_0$ from Lemma~\ref{J-2-pasos}~(iv) is essentially unique. To make this more precise, we need the following definition. 
\begin{definition}    
Let  $J$ be a complex structure on $\n$. We say that $(\n , J)$ {\em has no complex abelian factor} if  for any decomposition 
$\n=\a \oplus \n_1$ where $\a$ is a $J$-invariant abelian ideal  of $\n$ and 
 $\n _1$ is a $J$-invariant ideal of $\n$, then $\a =0$. \end{definition}
\begin{lemma}\label{lem:non-trivial-abelian} Let $J$ be a 2-step complex structure on a $2$-step nilpotent Lie algebra $\n$ and let $\z_0$ be a $J$-invariant ideal of $\n$ such that $\n'\subset \z_0\subset \z$.   Assume that $(\n ,J)$ has no  complex abelian factor. Then $\z_0=\n'+J\n'$. 
\end{lemma}
 \begin{proof} Since $\n'\subset \z_0$ and $\z_0$ is $J$-invariant we have that $\n'+J\n' \subset \z_0 $.  Assume that $\n'+J\n' \subsetneq \z_0 $, then there exists a non-zero complementary $J$-invariant subspace  $\a$ of $\n'+J\n'$ in $\z_0$. It follows that $\a$ is a complex abelian factor of $(\n ,J)$. In fact, $\a\subset \z_0\subset\z$, hence, $\a$ is an abelian ideal of $\n$ which is $J$-invariant. Moreover, let $\v$ be an arbitrary $J$-invariant subspace of $\n$ such that $\n=\z_0\oplus\v$, then  $\n_1=(\n'+J\n')\oplus \v$ is a $J$-invariant ideal of $\n$ and $\n=\a\oplus\n_1$ with $\a\neq 0$. This contradicts the fact that $(\n ,J)$ has no  complex abelian factor, and the lemma follows.
\end{proof}

We obtain next several consequences of Lemma \ref{J-2-pasos}. 

\begin{corollary}
\label{abelian-implies-2-pasos} 
Let $\n$ be a $2$-step nilpotent Lie algebra. If $J$ is an abelian or bi-invariant complex structure on $\n$, then $J$ is 2-step.
\end{corollary}
\begin{proof} Observe that if $J$ is abelian or bi-invariant, then $\z$ is $J$-invariant.  Since $\n$ is $2$-step nilpotent,  condition (ii) of Lemma \ref{J-2-pasos} is satisfied and the corollary follows. 
\end{proof}

 The following corollary gives an alternative proof of  \cite[Theorem 3.1]{GZZ}.
\begin{corollary}\label{cor:3-step} Let $\n$ be a $2$-step nilpotent Lie algebra with a complex structure $J$. If $J$ is 3-step  then 
$J\z \not \subset \z$.  
\end{corollary}
\begin{proof} Assume that, on the contrary, $J\z  \subset \z$. Then, since $\n'\subset \z$, we would have that  $J\n'  \subset \z$, and Lemma \ref{J-2-pasos} would imply that $J$ is 2-step, a contradiction.
\end{proof}


\ 


We denote by $\h _{2m+1}$  the $(2m+1)$-dimensional Heisenberg Lie algebra, which has a basis $\{ e_1, \ldots, e_m,  f_1, \ldots , f_m, z\}$  such that the only non-zero brackets are $[e_i,f_i]=-[f_i,e_i]=z$ for all $i$. This is a distinguished family of  2-step nilpotent Lie algebras.

We recall the following result, which is contained in \cite{Rol} (see  also  \cite[Proposition 2.2]{ABD}).
\begin{proposition}[{\cite[Proposition 3.6]{Rol}}] \label{prop:heisenberg}
If $\n$ is an even dimensional nilpotent Lie algebra with $\dim \n '=1$, then $\n$ is isomorphic to $\R^{2k+1}\oplus\h_{2n+1}$ and $\n$ admits a complex structure. Moreover, any complex structure on $\n$ is abelian. 
\end{proposition} 
 
\begin{corollary}\label{cor:heisenberg} If $\n$ is an even dimensional nilpotent Lie algebra with $\dim \n '=1$, any complex structure on $\n$ is 2-step. 
\end{corollary}
\begin{proof} Let $J$ be a complex structure on $\n$, Proposition \ref{prop:heisenberg} implies that $\n$ is isomorphic to $\R^{2k+1}\oplus\h_{2n+1}$, which is 2-step nilpotent,  and $J$ is abelian. The corollary follows from Corollary \ref{abelian-implies-2-pasos}.
\end{proof}

\section{Complex  structures on $2$-step nilpotent Lie algebras}

\begin{lemma}\label{idealJz} Let $J$ be a complex structure on a $2$-step nilpotent Lie algebra $\n$ with center $\z$ and 
commutator ideal $\n'$. If $\z_0$ is an ideal of $\n$ such that $\n' \subset \z_0 \subset \z$, then $J\z_0$ is an abelian ideal of $\n$. 
\end{lemma}
\begin{proof}
Let $J$ be a complex structure on $\n$ and $x,y\in \z_0$. Then 
\[ [Jx,Jy]=[x,y]+J[Jx,y]+J[x,Jy]=0,
\]
hence, $J\z_0$ is abelian. To show that it is an ideal, take $x\in \n$, $y\in\z_0$ and compute:
\[ [x, Jy]=J[x,y]-[Jx,y]-J[Jx,Jy]=-J[Jx,Jy],
\]
and since $[Jx,Jy]\in \n'\subset \z_0$, then $[x, Jy]\in J\z_0$.
\end{proof}

The above lemma motivates the study of abelian ideals in  $2$-step nilpotent Lie algebras (see Lemma \ref{ideal} below).  

The next simple lemma holds for arbitrary Lie algebras  (not necessarily nilpotent). 
\begin{lemma}\label{lem:ideal-in-center} Let $\g$ be a Lie algebra  and let $\a$  be an ideal of $\g$  such that $\a\cap\g'=0 $, then $\a \subset \z$.  In particular, $\a$ is abelian.
\end{lemma}
\begin{proof} Let $\a$ be as in the   statement, then:
\[ [\a , \g] \subset \a \cap \g ' =0, \; \text{ hence }\; \a \subset \z . \]
\end{proof}
The above lemma allows to show that non-zero abelian ideals of a $2$-step nilpotent Lie algebra have non-trivial intersection with its center. 

\begin{lemma}\label{ideal} Let $\n$ be a $2$-step nilpotent  Lie algebra  with center $\z$ and let 
   $\a \neq 0$ be  an abelian ideal of $\n$.  Then $\a\cap \z \neq 0$ and there are two possibilities:
\begin{enumerate}
\item[$\ri$] $\a\subset \z$, that is, $\a\cap \z = \a$,  
\item[$\rii$] $\a\not\subset \z$ and for any complementary subspace $\u  $ of $\a\cap \z$ in $\a$ it follows that $\u\cap \z= 0$, $\u$ is abelian and  $[\u,\n]\subset \a\cap \n '$. In particular,   $\a\cap \n ' \neq 0$. 
\end{enumerate}
\end{lemma}
\begin{proof} 
Let $\a \neq 0$ be  an abelian ideal of $\n$. We show first that  $\a\cap \z \neq  0$. 
  Assume that, on the contrary, 
$\a\cap \z = 0$. Then,  using that $\n '\subset \z$ (since $\n$ is $2$-step nilpotent),  we have 
$\a \cap \n'\subset \a \cap \z= 0$.    Lemma \ref{lem:ideal-in-center}  implies that  $\a \subset \z$, that is, 
$\a\cap \z = \a \neq 0$, a contradiction. 


It remains to show that, in case $\a\not\subset \z$, the conditions in (ii) are satisfied. In fact, 
if $\a\not\subset \z$ then $\a\cap \z \neq \a$ and let $\u \neq 0$ be any complementary subspace of $\a\cap \z$ in $\a$. 
Clearly, $\u \cap \z =0$ and $\u$ is abelian since $\u\subset\a$. 
   Since  $\a$ is an ideal, we have that $[\u,\n]\subset \a\cap \n '$. Moreover, $\u\cap \z= 0$, hence, $[\u,\n]\neq 0$, and (ii) follows.
\end{proof}

 A complex structure $J$ on a nilpotent Lie algebra $\n$ with center $\z$ is called \textit{strongly non-nilpotent} when $\a_1(J)= J\z \cap\z=0$    (see \cite{CFGU1}). 
If $\n$ is $2$-step nilpotent,  by applying Lemma \ref{ideal} with $\a=J\z$, which according to Lemma \ref{idealJz} is an ideal,   it follows that   $J\z \cap\z\neq0$, that is, $J$ is not strongly non-nilpotent. This gives a simple proof of an already known fact        
\cite[Corollary 3.6]{LUV}. 
\begin{corollary}[{\cite[Corollary 3.6]{LUV}}] \label{non-strongly} Let $\n$ be a $2$-step nilpotent Lie algebra. Then, any complex   structure $J$ on $\n$ satisfies $J\z \cap\z\neq 0$, that is, $J$ is not strongly non-nilpotent. In particular, $\dim \z\geq 2$.
\end{corollary}

\begin{corollary} \label{cor:n'2} Let $\n$ be a $2$-step nilpotent Lie algebra such that  $\dim \z=2$. Then, any complex   structure $J$ on $\n$ is 2-step. In particular, if $\n '=\z$ and  $b_1=\dim \n -2$, where $b_1$ is the first Betti number of $\n$, then any complex structure on $\n$ is 2-step.
    \end{corollary}
    \begin{proof}
        It follows from Corollary \ref{non-strongly} that $J\z \cap\z\neq 0$, therefore, $\z$ is $J$-invariant. Since $\n '\subset \z$, then $J\n '\subset J\z\subset \z$, that is, $J$ is 2-step. The last assertion follows since $b_1=\dim \n -2$ implies that $\dim \n '=2$, hence $ \dim \z= 2$.
    \end{proof}

From now on, given a complex structure $J$ on $\n$ we set   $\n'_J:= \n ' \cap J\n '$.   The following   fact will be needed in \S\ref{sec-3-step}. 
\begin{corollary}\label{cor:3-pasos}
    Let $J$ be a 3-step complex structure on a 2-step nilpotent Lie algebra $\n$. Then $\n '_J\neq 0$.  
    \end{corollary}
\begin{proof}
Let $J$ be a 3-step complex structure on  $\n$. Lemmas  \ref{J-2-pasos} and \ref{idealJz} imply that $J\n '$ is an abelian ideal not contained in $\z$.  The corollary follows from Lemma \ref{ideal} (ii) applied to the abelian ideal $\a=J\n '$.
\end{proof}

\

We show next that 3-step complex structures cannot occur on 2-step nilpotent Lie algebras with commutator ideal of dimension 1 or 2.
\begin{corollary}\label{cor:J-3-pasos}
    Assume that $\n$ is a 2-step nilpotent Lie algebra admitting a 3-step complex structure. Then $\dim\n '\geq 3$. 
\end{corollary}
\begin{proof}
    Assume that $J$ is a 3-step complex structure on $\n$, then there exists $z\in \n '$ such that $Jz\notin \z$ (Lemma \ref{J-2-pasos}). Hence, there exists $v\in\n$ such that $[Jz, v]\neq 0$ and since $N_J(Jz,v)=0$ we obtain that $J[Jz,v]=[Jz,Jv]$. In particular, $[Jz,v]$ and $[Jz,Jv]$ are linearly independent and $z\notin \text{span} \left\{ [Jz,v],  [Jz,Jv]\right\}$ since $Jz\notin \z$. Therefore, $\text{span} \left\{z,  [Jz,v],  [Jz,Jv]\right\}$ is a 3-dimensional subspace of $\n'$, and the corollary follows.
\end{proof}

As a consequence of the next result we will obtain  that there are families of 2-step nilpotent Lie algebras which can only admit 3-step complex structures (see Corollary~\ref{cor:free}~(ii) in \S\ref{sec-free}).
\begin{corollary}\label{cor:odd} Let $\n$ be a $2$-step nilpotent Lie algebra such that $\n '$ is odd dimensional and $\n'=\z$ . Then any complex structure on $\n$ is 3-step. 
\end{corollary}
\begin{proof} Assume that $J$ is a complex structure on $\n$. The abelian ideals $J\z$ and $\z\cap J\z$ are   odd  and   even dimensional, respectively, therefore, $J\z=\z\cap J\z\oplus \u $ for some $\u\neq 0$ such that $\u\cap\z=0 $ (Lemma \ref{ideal}). It follows that $J\n'=J\z=\z\cap J\z\oplus \u $ with $\u\neq 0$ and $\u\cap\z=0 $, in particular, $J\n'\not\subset \z  $. Lemma \ref{J-2-pasos} implies that $J$ is 3-step. 
\end{proof}

We give next examples of both, 2-step and 3-step complex structures.
\subsection{Complex structures on free 2-step nilpotent Lie algebras}  \label{sec-free}  We denote by $\f_r$ the free $2$-step nilpotent Lie algebra of rank  $r$, that is, given a real $r$-dimensional vector space $V$, set $\f_r = V \oplus \Lambda^2\left( V\right)$ with center 
$\Lambda^2\left( V\right)$ and Lie bracket  $[v,w]= v\wedge w$, for $ v,w \in V$. Note that    $\f_2$ is the 3-dimensional Heisenberg Lie algebra $\h_3$. 

We show next that $\f_r$ always admits a complex structure when $\dim\f_r$ is even. On the other hand, if $\dim\f_r$ is odd, then $\R\oplus \f_r$ admits a complex structure.  
We point out that $\dim\f_r$ is even if and only if   $r\equiv 0$ or $3  \pmod 4$.   
\begin{proposition} \label{prop:free}
\begin{enumerate} \item[]
    \item [$\ri$] 
If $r\equiv 0$ or $3  \pmod 4$, then $\f_r$ admits a complex structure.
\item [$\rii$] 
If $r\equiv 1$ or $2 \; \pmod 4$, then $\R\oplus \f_r$ admits a complex structure.
\end{enumerate}

\end{proposition}
\begin{proof}
    (i) Assume that $r\equiv 0$ or $3  \pmod 4$.  
 We study each case separately.

\begin{itemize}
    \item $r\equiv 0 \pmod 4$. Given a basis  $v_1, \ldots, v_{2k}, w_1, \ldots , w_{2k}$  of $V$, consider the following basis of $\Lambda^2\left( V\right)$:
    \begin{eqnarray*}
        \alpha_{ij}^\pm &:=&  w_i\wedge v_j\pm v_i\wedge w_j ,\quad \text{ for } i<j,\\
        \beta_{ij}^\pm &:=& w_i\wedge w_j\pm v_i\wedge v_j ,\quad \text{ for } i<j, \\
        \gamma_i&:=& v_i\wedge w_i, \qquad     i=1, \ldots , 2k.
    \end{eqnarray*}
We define a complex endomorphism $J$ of $\f_r$ as follows:
\begin{equation}
 Jv_i=w_i, \qquad    J\alpha_{ij}^+=\beta_{ij}^-, \qquad J\alpha_{ij}^-=\beta_{ij}^+, \qquad J\gamma_{2i-1}= \gamma_{2i}. 
\end{equation}
It turns out that $J$ is a 2-step complex structure on $\f_r$, that is, $N_J\equiv 0$ and $J[\f_r ,\f_r]\subset \z$.  
    \item $r\equiv 3 \pmod 4$. Given a basis  $v_0, v_1, \ldots, v_{2k-1}, w_1, \ldots , w_{2k-1}$  of $V$, consider the following basis of $\Lambda^2\left( V\right)$:
    \begin{eqnarray*}
        \alpha_{ij}^\pm &:=&  w_i\wedge v_j\pm v_i\wedge w_j ,\quad \text{ for } 1\leq i<j\leq 2k-1,\\
        \beta_{ij}^\pm &:=& w_i\wedge w_j\pm v_i\wedge v_j ,\quad \text{ for } 1\leq i<j\leq 2k-1, \\
        \gamma_i&:=& v_i\wedge w_i, \qquad     i=1, \ldots , 2k-1,\\
        \epsilon_i&:=& v_0\wedge v_i, \qquad     i=1, \ldots , 2k-1,\\
        \delta_i&:=& v_0\wedge w_i, \qquad     i=1, \ldots , 2k-1.
    \end{eqnarray*} 
We define a complex endomorphism $J$ of $\f_r$ as follows:
\begin{equation}
 \begin{split}     
 \!\!\!Jv_0= \gamma_{2k-1}, \quad\qquad Jv_i=w_i, \quad i\geq 1, \qquad  \quad  J\alpha_{ij}^+=\beta_{ij}^-,\\ \qquad J\alpha_{ij}^-=\beta_{ij}^+, \qquad J \epsilon_i = \delta_i, \qquad J\gamma_{2i-1}= \gamma_{2i}, \quad i=1, \ldots, k-1. \end{split}
\end{equation}
It turns out that $J$ is a 3-step complex structure on $\f_r$, since $J[\f_r ,\f_r]$ is not contained in the center of $\f_r$.
\end{itemize}

(ii) If $\dim\f_r$ is odd, the proof is analogous to (i) by considering  separately the cases $r\equiv 1 \pmod 4$ and $r\equiv 2 \pmod 4$.  
\begin{itemize}
    \item $r\equiv 1 \pmod 4$. Given a basis  $v_0, v_1, \ldots, v_{2k}, w_1, \ldots , w_{2k}$  of $V$, consider the following basis of $\Lambda^2\left( V\right)$:
    \begin{eqnarray*}
        \alpha_{ij}^\pm &:=&  w_i\wedge v_j\pm v_i\wedge w_j ,\quad \text{ for } i<j,\\
        \beta_{ij}^\pm &:=& w_i\wedge w_j\pm v_i\wedge v_j ,\quad \text{ for } i<j, \\
        \gamma_i&:=& v_i\wedge w_i, \qquad     i=1, \ldots , 2k,\\
    \epsilon_i&:=& v_0\wedge v_i, \qquad     i=1, \ldots , 2k,\\
        \delta_i&:=& v_0\wedge w_i, \qquad     i=1, \ldots , 2k.
    \end{eqnarray*}
We define a complex endomorphism $J$ of $\R w_0\oplus\f_r$ as follows:
\begin{equation}
 Jv_i=w_i, \qquad    J\alpha_{ij}^+=\beta_{ij}^-, \qquad J\alpha_{ij}^-=\beta_{ij}^+, \qquad J\gamma_{2i-1}= \gamma_{2i}, \quad
 J \epsilon_i = \delta_i.
\end{equation}
It turns out that $J$ is a 2-step complex structure on  $\R w_0\oplus\f_r$. 

\smallskip

    \item $r\equiv 2 \pmod 4$. Given a basis  $ v_1, \ldots, v_{2k+1}, w_1, \ldots , w_{2k+1}$  of $V$, consider the following basis of $\Lambda^2\left( V\right)$:
    \begin{eqnarray*}
        \alpha_{ij}^\pm &:=&  w_i\wedge v_j\pm v_i\wedge w_j ,\quad \text{ for }  i<j,\\
        \beta_{ij}^\pm &:=& w_i\wedge w_j\pm v_i\wedge v_j ,\quad \text{ for }  i<j, \\
        \gamma_i&:=& v_i\wedge w_i, \qquad     i=1, \ldots , 2k+1.
    \end{eqnarray*} 
We define a complex endomorphism $J$ of $\R w_0\oplus\f_r$ as follows:
\begin{equation}
 Jv_i=w_i, \qquad    J\alpha_{ij}^+=\beta_{ij}^-, \qquad J\alpha_{ij}^-=\beta_{ij}^+, \qquad J\gamma_{2i-1}= \gamma_{2i}, \quad J \gamma_{2k+1}=w_0.
\end{equation}
It turns out that $J$ is a 2-step complex structure on $\R w_0\oplus\f_r$. This concludes the proof of (ii), and the proposition follows.
\end{itemize}

\end{proof}

\begin{remark}
    In contrast with  Proposition \ref{prop:free}, it follows from \cite[Theorem 3]{PT} that  $\R^d\oplus \f_r$ admits a symplectic structure if and only if $d=0$ and $r=3$ or $d=1$  and $r=2$. 
\end{remark}

\begin{remark}    
    If  $r\equiv 0 \pmod 4$, the complex structure on $\ \f_r$  exhibited in Proposition~\ref{prop:free} is 2-step.  When $r\equiv 3 \pmod 4$,   the corresponding complex structure  is 3-step. Moreover, we will show in Corollary \ref{cor:free} below that if $r\equiv 0 \pmod 4$ (resp., $r\equiv 3 \pmod 4$), then  any complex structure on $ \f_r$ is 2-step (resp., 3-step).
\end{remark}

Corollary \ref{cor:free} will follow from the fact that, 
 if $\a$ is an abelian ideal of $\R^d\oplus\f_r$ which is not contained in the center,  then $\a\cap \z$ has codimension $1$ in $\a$.
\begin{lemma}\label{lem:ideal-free} Let $\f_r$ be the free $2$-step nilpotent Lie algebra of rank  $r$ and let $d$ be a nonnegative integer. 
\begin{enumerate}
\item[$\ri$] If $\u$ is an abelian subalgebra of  $\R^d\oplus\f_r$ such that $\u\cap \z =0 $, then $\u=0$ or $\dim \u=1$. 
\item[$\rii$] If $\a$ is an abelian ideal of  $\R^d\oplus\f_r$ such that $\a\not \subset \z$, then $\a\cap \z$ has codimension $1$ in~$\a$.
\end{enumerate}

\end{lemma}

\begin{proof} To prove $\ri$, assume that, on the contrary,  $\dim \u >1$. Then there exist linearly independent $u, v\in \u$. Decompose $u=u_1+u_2, \,
v=v_1+v_2$ with $u_1, v_1\in \z$ and $u_2, v_2\in V,$ hence $u_2, v_2$ are linearly independent, since $\u\cap \z =0 $. Therefore, $u_2\wedge v_2 \neq 0$ and 
$[u,v]=[u_2, v_2]=u_2\wedge v_2 \neq 0$, a contradiction, since $\u$ is abelian.

Let $\a$ be an abelian ideal of $\R^d\oplus\f_r$ such that $\a\not \subset \z$, then  
 Lemma \ref{ideal} applies, so   there exists an abelian subalgebra $\u\neq 0$ such that $\a= \a\cap \z \oplus\u$ with $\u\cap \z =0$. Therefore,   (i) implies that $\dim \u =1$ and (ii) follows.
\end{proof}

\begin{corollary}\label{cor:free} Let $\f_r$ be the free $2$-step nilpotent Lie algebra of rank  $r$.
\begin{enumerate} 
\item[$\ri$]  If $r\equiv 0 \pmod 4$ then any complex structure on $\f_r$ is 2-step.
\item[$\rii$]  If $r\equiv 3 \pmod 4$ then  any complex structure $J$ on $\f_r$ is 3-step. Moreover, $\z\cap J\z$ has codimension $1$ in $\z$.
\end{enumerate}
\end{corollary}
\begin{proof} The Lie algebra $\f_r$ satisfies $\f_r'=\z = \Lambda^2\left( V\right) $. 

If $r\equiv 0 \pmod{4}$ then $\dim\z$ is even, therefore both, $J\z$ and  $\z\cap J\z$ are even dimensional. Lemma \ref{lem:ideal-free} implies that the abelian ideal $J\z$ must satisfy $J\z\subset \z$, that is, $J\f_r' \subset \z$, in other words, $J$ is 2-step  (Lemma \ref{J-2-pasos}) and (i) follows.

If $r\equiv 3 \pmod{4}$ then $r=2s+1$ with $s\equiv 1 \pmod{2}$,   therefore, 
$\dim\f_3'=\dim\z = (2s+1)s \equiv 1 \pmod{2}$  and  Corollary \ref{cor:odd} implies that any complex structure on $\f_r$ is 3-step. Given a complex structure $J$ on $\f_r$, the abelian ideal $J\z$ must satisfy $J\z\not \subset\z$ (Corollary \ref{cor:3-step}). Lemma \ref{lem:ideal-free} implies that 
$J\z= \z\cap J\z \oplus\u$ with $\u\cap \z =0 $ and  $\dim \u=1$, hence, $\z= \z\cap J\z \oplus J\u$ and the last assertion in (ii)   follows. 
\end{proof}

\subsection{Six dimensional 2-step nilpotent Lie algebras} We end this section by studying the 6-dimensional case.  
The list of six dimensional real $2$-step nilpotent Lie algebras is given  in Table \ref{tabla} (see \cite{Mag,Mor}). 
We recall Salamon's notation for an $m$-dimensional  Lie algebra $\n$. Let $e^1, \dots , e^m$ be a basis of $\n ^*$, then $\n$ is represented   by an $m$-tuple $(de^1, \dots , d e^m)$, and if 
$d e^i=\sum_{j<k}c^i_{jk} e^j\wedge e^k$, then $e^j\wedge e^k$ is abbreviated as $jk$. For instance, the 6-tuple $(0,0,0,13-24,14+23)$    denotes the 6-dimensional  Lie algebra defined by: 
\[ d e^5= e^{1}\wedge e^3-e^{2}\wedge e^4, \qquad\quad d e^6= e^{1}\wedge e^4+e^{2}\wedge e^3.
\]
As a consequence of the results  in \cite{Sal} it follows that all Lie algebras in Table \ref{tabla} admit complex structures. 
Any complex structure on the Lie algebra $\f_3$ is 3-step  (Corollary \ref{cor:free}), while  the remaining 6-dimensional 2-step nilpotent Lie algebras only admit 2-step complex structures.

\begin{table}[ht] 
\begin{tabular}{lccc} \hline
 $\qquad\qquad \n$ & $\dim\n '$ & $\dim \z$ 
 \\ \hline
 $(0,0,0,12,13,23)$ & $3$ &  $3$ & $\f_3$ \\
 $(0,0,0,0,13-24,14+23)$ & $2$ & $2$ & $\h_3(\C )$\\
$(0,0,0,0,12,14+23)$ & $2$ & $2$ & $\h_3\ltimes\R^3$  \\
$(0,0,0,0,12,34)$ & $2$ & $2$ & $\h_3 \oplus \h_3 $ \\
$(0,0,0,0,12,13)$ & $2$ & $3$ & $ \R\ltimes \R^5$ \\
$(0,0,0,0,0,12+34)$ & $1$ & $2$ & $\R \oplus \h_5 $  \\
$(0,0,0,0,0,12)$ & $1$ & $4$  & $\R^3 \oplus \h_3 $ \\
\hline
\vspace*{.1cm}
\end{tabular} 
\vspace*{.1cm}
\caption{\label{tabla} Six dimensional real $2$-step nilpotent Lie algebras.}
\end{table}

\begin{proposition}
    Any 6-dimensional 2-step nilpotent Lie algebra $\n$ admits a complex structure. Moreover, $\n$ admits  a 3-step complex structure  if and only if $\n$ is isomorphic to $\f_3$, and any complex structure on $\f_3$ is 3-step.
\end{proposition}

\begin{proof}
    The first assertion follows from \cite{Sal}.  If $\n$ 
    admits a 3-step complex structure, then $\dim\n' \geq 3$ (Corollary \ref{cor:J-3-pasos}), therefore,  $\n$ is isomorphic to $\f_3$ (see Table \ref{tabla}). Finally, Corollary \ref{cor:free} implies that any complex structure on $\f_3$ is 3-step.
\end{proof}

\smallskip

Our aim is to give a characterization of the $2$-step nilpotent Lie algebras $\n$ admitting a complex structure, so let $J$ be a complex structure on  $\n$. According to Theorem \ref{2or3}, $J$ is $t$-step for $t=2$ or $3$.  
We study each case separately and we will show that in the latter case the complex structure can be constructed by combining suitably an abelian complex structure with a 2-step complex structure.

\section{Two-step complex structures} 

Let $\n$ be a 2-step nilpotent Lie algebra with an inner product $\pint$ and fix an ideal $\z _0$ of $\n $ such that $\n '\subset \z _0 \subset \z$. Let $\v$ be the orthogonal complement of $\z_0$ in $\n$.   
 Since $\n'\subset \z_0$, the Lie bracket on $\n$ amounts to a linear map $j: \z_0 \rightarrow \operatorname{End}\, (\v)$, 
$ z \mapsto j(z)$, where 
 $j(z)$, $z\in \z_0$, is defined as follows:
\begin{equation} \label{jz}
\langle j(z)v,w \rangle =
\langle z, [v,w]\, \rangle , \;\;\; \;\forall v,w \in \v.
\end{equation} 
Observe that $j(z),\; z\in \z_0$,  are skew-symmetric so that 
$j$ defines 
 a linear map $j: \z_0 \rightarrow \s \o (\v, \pint)$. Note that Ker$(j)$ is the orthogonal complement of 
 $\n'$ in $\z_0$.   The center $\z$ of $\n$ is given by $\z=\z_0\oplus \bigcap _{z\in \z_0}\, \text{Ker}\, j(z)$.

Let $J$ be a 2-step complex structure  on  $\n$.  Assume that $(\n, J)$ has no  complex abelian factor and let  
 $\z_0=\n '+J\n'\subset \z$  (see  Lemma \ref{lem:non-trivial-abelian}). 
  We fix an auxiliary Hermitian inner product $\pint$ on $(\n ,J)$. Note that 
$\n'$ is $J$-invariant  if and only if $j: \z_0 \rightarrow \s \o (\v)$ defined in \eqref{jz} is injective. 
We define a linear map $S:\z_0 \to \operatorname{End}(\v)$ as follows:
\begin{equation}\label{S_z} S(z)= j(Jz) - J_\v \circ j(z), \quad z\in \z_0 .
\end{equation}
We will show in Proposition \ref{main1} below  that the integrability of $J$ is equivalent to the fact that the image of $S$ lies in $\mathfrak{gl}(n,\C)$, where $\dim\v =2n$.

 The proof of the next lemma   is straightforward. 
\begin{lemma}\label{lem_S} The linear map $S:\z_0 \to \operatorname{End}(\v)$ defined in \eqref{S_z} satisfies: 
\[ S({Jz})=-J_\v \circ S(z) \quad \text{ for all } z\in \z_0 .\]
In particular, $\ker S$ is $J$-invariant.
\end{lemma}
The next proposition gives a necessary and sufficient condition, in terms of the map $S$ defined in \eqref{S_z}, for the integrability of a complex structure $J$  satisfying $J\n '\subset \z$.

\begin{proposition}\label{main1} Let $\n$ be a $2$-step nilpotent Lie algebra with an almost complex structure~$J$ such that $J\n '\subset \z$ and let $\z_0:=\n '+J\n$. 
Fix an arbitrary Hermitian inner product $\pint$ on $\n$,  decompose $\n$ orthogonally as $\n=\z_0\oplus \v$ and let $S:\z_0 \to \operatorname{End}(\v)$ be the linear map defined    in \eqref{S_z}. Then $J$ is integrable if and only if $S(z)$ commutes with  $J_\v$ for all $z\in \z_0$, in other words, 
\begin{equation}\label{integ2}\operatorname{Image}(S)\subset \mathfrak{gl}(n,\C)=\{ T\in \operatorname{End}(\v) : T  \circ J_\v = J_\v \circ T\},  \end{equation} 
where $\dim\v =2n$.  
\end{proposition}
\begin{proof}
For $z\in \z _0, \; u,v \in \v$, we compute
\[ \la j(Jz)u,v\ra=\la Jz,[u,v]\ra=-\la z, J[u,v]\ra=-\la z, [Ju,v]+[u,Jv]+J[Ju,Jv]\ra,
\]
where we have used the integrability of $J$. The last term above can be written as
\begin{eqnarray*} & & -\la j(z)Ju,v\ra-\la j(z)u,Jv\ra +\la Jz, [Ju,Jv]\ra, \\ 
&=& -\la j(z)Ju,v\ra-\la j(z)u,Jv\ra +\la j(Jz)Ju, Jv\ra .
\end{eqnarray*}
Therefore,  
\[ -\la j(Jz)u,v\ra+\la j(Jz)Ju, Jv\ra=\la j(z)Ju,v\ra+\la j(z)u,Jv\ra, 
\] 
and using that $J$ is skew-symmetric on the second term of both sides we obtain:
\[ -\la j(Jz)u,v\ra -\la Jj(Jz)Ju, v\ra=\la j(z)Ju,v\ra-\la Jj(z)u,v\ra .
\]
Since this holds for all $u, v \in \v$,  
the above equation is equivalent to 
\begin{equation}\label{integ}
j(Jz)+J_\v \circ j(Jz)\circ J_\v = J_\v \circ j(z)-j(z)\circ J_\v= [J_\v,j(z)], \quad \text{for all }z\in\z_0 ,
\end{equation}
where the Lie bracket above is the usual one on $\text{End} (\v)$.    It is straightforward that the integrability of $J$ on $\n$ is equivalent to \eqref{integ}. The left-hand side of \eqref{integ} can be written as $J_\v \circ [j(Jz), J_\v ]$. Therefore, $J$ is integrable if and only if 
\begin{equation}\label{integ1}
[J_\v , j(Jz)]=- J_\v \circ [j(z), J_\v ]= [J_\v , J_\v \circ j(z)], \quad\quad \text{for all }z\in\z_0 ,
\end{equation}
where the last equality follows by direct computation. 
Equation \eqref{integ1} means that $S(z)=j(Jz) - J_\v \circ j(z)$ commutes with $J_\v$ or, equivalently, 
$S(z)\in \mathfrak{gl}(n,\C)$ for all $z \in \z_0$,  
and the proposition follows.  
\end{proof}

The next corollary is a straightforward consequence of Proposition \ref{main1} and Lemma \ref{lem_S}.
\begin{corollary}\label{corol_main1} Let $(\n,J,\pint)$, $\z_0$, $\v$ and $S$ be as in Proposition \ref{main1} and let $\p\subset \z_0$ be an arbitrary subspace such that $\z_0=\p\oplus J\p$. Then 
$J$ is integrable if and only if $S(z)\in \mathfrak{gl}(n,\C)$ for all $z\in \p$. 
\end{corollary}

Let $(\n,J,\pint)$, $\z_0$, $\v$ and $S$ be as in Proposition \ref{main1}.  
There are two possibilities:
\begin{enumerate} 
		\item[$\ri$]  $\n'$ is $J$-invariant, that is, $\z_0=\n '$,
		\item[$\rii$] $\n'_J=\n'\cap J\n'\subsetneq \n '$ (possibly  
 $\n'_J=0$). \end{enumerate}
In case (ii), fix 
\begin{equation}\label{z_1} \z_1\subset \n' \;\text{ such that }\; J\z_1 \perp\n' \;\text{ and  }\; \n'=\z_1\oplus \n'_J. 
\end{equation}
In case $\n'_J=0$, then $\z_1=\n '$. 

We have the orthogonal decomposition 
\begin{equation}\label{eq:dec-z_0}
    \z_0= J\z_ 1\oplus \underbrace{\z_1\oplus \n'_J}_{\n'} , \quad \text{with } \z_1=0 \;\text{ if $\n '$ is $J$ invariant}.
\end{equation}
		Since $ J\z_1$ is the orthogonal complement of $\n'$ in $\z_0$,  $\ker(j)= J\z_1$.
Let $\v_0\subset \v$ be a subspace such that $\v$ decomposes orthogonally as $\v=\v_0\oplus J\v_0$ and define an involution $\sigma: \operatorname{End}(\v)\to \operatorname{End}(\v)$ as follows:
\[ \sigma (T)=P\circ T, \quad \text{ where } P|_{\v_0}=I,\quad P|_{J\v_0}=-I.
\] 
This induces a decomposition 
\begin{equation}\label{decomp} \operatorname{End}(\v)= \mathfrak{gl}(n,\C) \oplus \sigma \left( \mathfrak{gl}(n,\C) \right),
\end{equation}
so given $T\in \operatorname{End}(\v)$ we can write 
\begin{equation}\label{components}
T=T_+ +T_- \, , \quad \text{ with } \;  T_+\in \mathfrak{gl}(n,\C), \quad T_-\in \sigma \left( \mathfrak{gl}(n,\C) \right).\end{equation} 
Note that $T\in \mathfrak{gl}(n,\C)$ if and only if $T=T_+$. 
If $T=\begin{pmatrix} T_1&T_2\\T_3&T_4    
\end{pmatrix}$ with respect to the decomposition $\v=\v_0\oplus J\v_0$, then $T_+$ and $T_-$ are given by
\begin{equation}\label{eq:T+-}
    T_+=\frac 12\begin{pmatrix} T_1+T_4&T_2-T_3 \\ T_3-T_2&T_1+T_4        
    \end{pmatrix}, \qquad\quad  T_-=\frac 12\begin{pmatrix} T_1-T_4&T_2+T_3 \\ T_2+T_3&T_4-T_1        
    \end{pmatrix}.
\end{equation}
Note that \[ \sigma \left( \mathfrak{gl}(n,\C) \right)=\left\{ T\in \operatorname{End}(\v): T\circ J=-J\circ T \right\}. \] Moreover, if $T\in \s\o(2n )$, then $T_\pm \in \s\o(2n )$, therefore, it turns out that 
\[ \s \o (2n)=  \mathfrak{u}(n)\oplus \sigma \left( \mathfrak{gl}(n,\C) \right) \cap \s \o (2n).  \]

The proof of the following lemma is straightforward.
\begin{lemma}\label{lem:+-}
    Let $T, J\in \operatorname{End}(\v)$ such that $J^2=-I$ and let $\v_0$ be a subspace of $\v$ such that $\v =\v_0\oplus J\v_0$. Then, $(J\circ T)_\pm = J\circ \left( T_\pm\right)$, where $T_\pm$ are given by \eqref{components}.
\end{lemma}
We introduce some notation that will be used in the next proposition. Let $j: \z_0\to \s\o(2n)$  as defined in \eqref{jz} and set \begin{equation}\label{eq:j+-}
    \begin{split} 
j_+ :\z_0\to \u(n),\qquad\quad   j_-: \, & \z_0 \to \sigma \left( \mathfrak{gl}(n,\C) \right) \cap \s \o (2n),\\
z\mapsto j(z)_+ , \;\qquad\quad\quad   & z\mapsto j(z)_- .
\end{split} \end{equation}
 In case $\n'_J\neq 0$,  decompose  
 $ \n'_J= \b\oplus J\b$ orthogonally for some non-zero subspace $\b$. 
\begin{proposition}\label{prop:main2} Let $(\n,J,\pint)$, $\v$ and $\z_0$  be as in Proposition \ref{main1}, $\z_1$ as in \eqref{z_1} and $\b$ as above. The following statements are equivalent:
\begin{enumerate}
\item[$\ri$] $J$ is integrable, 
\item[$\rii$] $j(z) \in \mathfrak{u}(n)$ for all $z\in\z_1$ and $j(Jz)-\frac12\left[ J_\v ,j(z)\right]\in \mathfrak{u}(n)$ for all $z\in \b$,
\item[$\riii$] $j(z) \in \mathfrak{u}(n)$ for all $z\in\z_1$ and $j_-(Jz)= \frac12\left[ J_\v ,j(z)\right] $ for all $z\in \b$ $\rlp$see \eqref{eq:j+-}$\rrp$. 
\end{enumerate}
\end{proposition}
\begin{proof}
Since $j(z), \, j(Jz), \, J_\v\in \s \o (2n)$ for $z\in \z_0$, the symmetric and skew-symmetric parts   of $S(z)$ are given, respectively, by:
\begin{equation}\label{S^a} S(z)_{sym}=-\frac12\left\{ J_\v ,  j(z)\right\} , \qquad\quad S(z)_{skew}=j(Jz)-\frac12[J_\v ,j(z)],  
\end{equation}
where $\{A,B\}=A\circ B +B\circ A$ is the anticommutator of the operators $A$ and $B$. 
Note that $S(z)_{sym}\in \mathfrak{gl}(n,\C)$, hence,  
\[ S(z)\in \mathfrak{gl}(n,\C) \;\text{ if and only if }\; S(z)_{skew} \in \mathfrak{gl}(n,\C) ,\] 
and therefore, 
\[S(z)_{skew} \in \mathfrak{gl}(n,\C) \cap \s \o (2n)=  \mathfrak{u}(n) \; \text{ for all } z\in\z_0. \]
It follows that:
\begin{equation}\label{char} j(Jz)=\frac12\left[ J_\v ,j(z)\right]+ S(z)_{skew}, \quad \text{ where }  S(z)_{skew} \in\mathfrak{u}(n), \; \text{ for all } z\in\z_0.
\end{equation}
Note that $\ker j=  J\z_1 $ so, in view of \eqref{char}, $j(z)\in \mathfrak{u}(n)$ for all $z \in \z_1$. 
\end{proof}

The following lemma  follows  from \cite[Lemma 1.1]{Bar} and  \cite[Remark 4.8]{BM}. 
\begin{lemma}\label{rem:abelian}\cite{Bar, BM} Let $J$ be a complex structure on a 2-step nilpotent Lie algebra $\n$. Fix a  Hermitian inner product $\pint$ on $(\n ,J)$ and set $\z_0=\n '+J\n '$.  Let $\v$ be the orthogonal complement of $\z_0$ in $\n$ and let $j:\z_0 \to \s\o (\v)$ be as in \eqref{jz}.  Then \begin{enumerate}
 \item[$\ri$] 
 $J$ is abelian if and only if the center $\z$ is $J$-invariant and $j(z)\circ J_\v =J_\v\circ j(z)$ for all $z\in\z_0$.
 \item[$\rii$]  $J$ is bi-invariant if and only if the commutator $\n '$ is $J$-invariant and $j(z)\circ J_\v = -J_\v\circ j(z)$ for all $z\in\n '$.
 \end{enumerate}
\end{lemma}

\begin{remark} The above lemma also holds for $\z_0=\z$,  $\v = \z^\perp$  and $j:\z \to \s\o (\v)$ defined  in   \eqref{jz}.    
\end{remark}

Let $(\n,J,\pint)$, $\z_0$, $\v$ and $S$ be as in Proposition \ref{main1} and assume that $J$ is integrable.   
Given a $J$-invariant subspace $\q$ of $\z_0=n '+J\n '$, we define
\[     
 \q _\pm :=\left\{ z\in \q : j(z)=j_\pm(z)\right\}. \]

 
 We show next that  $\q _+$ is  $J$-invariant. On the other hand,    $\q _-$ is not  $J$-invariant unless $\ker \left(S|_{\q} \right) =\q _-$.  Indeed, $\ker \left(S|_{\q} \right)$ is the maximal $J$-invariant subspace of $\q _-$.
\begin{lemma}\label{prop:J-inv} Let $(\n,J,\pint)$ be a Hermitian Lie algebra such that $\n$ is 2-step nilpotent and  $J$ is a  2-step complex structure on $\n$. Let $\q$ be a $J$-invariant subspace of $\z_0$.  Then
\begin{enumerate}
    \item [$\ri$]   $\q _+$ is a $J$-invariant   subspace of $\q$,
    \item [$\rii$] $\ker \left(S|_{\q} \right) \subset \q _-$, 
    \item [$\riii$] $\ker \left(S|_{\q} \right)$ is the maximal $J$-invariant subspace of $\q _-$. In particular, if $\q _-$ is $J$-invariant, then $\q _- =\ker \left(S|_{\q} \right)$. 
\end{enumerate}
\end{lemma}
\begin{proof}
    It is clear from \eqref{char} that, if  $z\in \z_0$,  then $j(z)\in \u(n)$ if and only if $j(Jz)\in \u(n)$. In particular, $\q _+$ is $J$-invariant and (i) follows. 


Note that for $z\in \z_0$, \eqref{char} implies that $z\in \rlp\z_0\rrp_-$ if and only if 
$S(Jz)_{skew}=0$. Since $S(Jz)_{skew}=0$ if and only if $S(z)$ is skew-symmetric, we obtain 
\begin{equation}\label{eq:kerS} \ker S = \{ z\in \z_0 :  z, \, Jz\in \rlp\z_0\rrp_- \}, \end{equation}
and (ii) follows.

If $W$ is a $J$-invariant    subspace of $\q _-$,  then each $w\in W$ satisfies that both, $w$ and $Jw$ lie in $W\subset \q_-$, hence, \eqref{eq:kerS} implies that $w\in \ker \left(S|_{\q} \right)$. Therefore, 
$W\subset \ker \left(S|_{\q} \right)$, and the proof of (iii) is complete.

\end{proof}

It was proved in \cite[Lemma 3.1 (viii)]{Rol} that if $J$ is a complex structure on a  nilpotent Lie algebra $\n$ such that $\n '$ contains no non-trivial $J$-invariant subspace, then $J$ is abelian. In other words, if $\n ' _J=0$, then $J$ is abelian. We show next that the converse does not hold for 2-step nilpotent Lie algebras. 

\begin{corollary}\label{abelian}  Let $J$ be a complex structure on a $2$-step nilpotent Lie algebra $\n$. Then  \begin{enumerate}
    \item[$\ri$] 
$J$ is abelian if and only if $\z$ is $J$-invariant and   $\n'_J=\rlp\n_J'\rrp _+$. 
\item[$\rii$] $J$ is bi-invariant if and only if any of the  following equivalent conditions is satisfied:
\begin{enumerate}  \item [$\rm{(a)}$]  $\n '=\rlp\n_J'\rrp _-$, 
\item [$\rm{(b)}$]  the map $j: \z_0 \to \operatorname{End}(\v)$ is complex linear, where $\z_0$ and $\operatorname{End}(\v)$ are regarded as complex vector spaces via $J|_{\z_0}$ and $\mathcal L _{J_\v}: \operatorname{End}(\v)\to \operatorname{End}(\v)$, $T \mapsto J_\v \circ T$, respectively. 
\end{enumerate}
\end{enumerate}
\end{corollary}

\begin{proof} To prove (i), assume first that $J$ is abelian. 
According to Lemma~\ref{rem:abelian} (i), $\z$ is $J$-invariant and $j(z)$ commutes with $J_\v$ for all $z\in\z_0$, that is, $j(z)\in\mathfrak{u}(n) $ for all $z\in\z_0$. In particular, $j(z)\in\mathfrak{u}(n) $ for all $z\in\n '_J$, that is, $\n'_J=\rlp\n_J'\rrp _+$. Conversely, if $\z$ is $J$-invariant and $\n'_J=\rlp\n_J'\rrp _+$, then $\n'=\z_1\oplus \rlp\n_J'\rrp _+$ (see \eqref{z_1}), $\z_0= J\z_1\oplus \z_1\oplus \rlp\n_J'\rrp _+$ and 
 \eqref{char} implies that $j(z)\in \mathfrak{u}(n) $ for all $z\in\z_0$, since   $ J\z_1=\ker (j)$. Lemma~\ref{rem:abelian} (i) implies that $J$ is abelian, and (i) follows.

 Assume next that  $J$ is bi-invariant,  then Lemma~\ref{rem:abelian} (ii) implies that $\n '$ is $J$-invariant and $j(z)$ anticommutes with $J_\v$ for all $z\in \n '$. Hence, $\n'=\n'_J$ and $j(z)$ anticommutes with $J_\v$ for all $z\in \n '_J$, that is, 
 $\n '_J=\rlp\n_J'\rrp _-$. Therefore, $\n '=\rlp\n_J'\rrp _-$. 
 Conversely, if $\n '=\rlp\n_J'\rrp _-$, then $\z_1=0$, that is, $\n'= \n '_J$. Therefore,    $\n'$  is $J$-invariant and    $j(z)$ anticommutes with $J_\v $ for all $z\in \n '$, hence, $J$ is bi-invariant (see Lemma~\ref{rem:abelian} (ii)). 

Finally, we show that (a) and (b) are equivalent. Condition (a) is equivalent to $\z_0=\n'$ and $j(w)=j_-(w)$ for all $w\in \z_0$. In particular, 
  $j(z)=j_-(z)$ and $j(Jz)=j_-(Jz)$  for all $z\in \z_0$. 
 Recall from \eqref{eq:kerS} that  this holds if and only if $S(z)=0 $. This is equivalent to $j(Jz)=J_\v\circ j(z)$, that is,   $j:  \z_0 \to \operatorname{End}(\v)$ is complex linear.

\end{proof}


Let $(\n,J,\pint)$, $\v, \z_0$ and $\z_1$ be as in Proposition \ref{prop:main2} and assume that $J$ is integrable.  If $\n'_J\neq 0$, decompose  
 \begin{equation}\label{eq:refine}
     \n'_J= \rlp\n_J'\rrp _+\oplus  \ker \left(S|_{\n'_J} \right) \oplus \a, \qquad \a\,  \perp \left( \rlp\n_J'\rrp _+\oplus  \ker \left(S|_{\n'_J} \right) \right) ,\end{equation}  
 where   some of these subspaces could be zero. 
 In view of Lemma \ref{lem_S} and Lemma \ref{prop:J-inv} (i), we have that   
 \eqref{eq:refine} is a decomposition of $\n'_J$  into $J$-invariant  subspaces. We chose subspaces $\p_+ \subset \rlp\n_J'\rrp _+, $
$\p_-\subset\ker \left(S|_{\n'_J} \right)$ and  $\a_1\subset \a$ such that we have the following orthogonal decompositions
\[  \rlp\n_J'\rrp _+=\p_+\oplus J\p_+, \qquad 
\ker \left(S|_{\n'_J} \right)= \p_-\oplus J\p_-, \qquad \a=\a_1\oplus J\a_1 .
\] Set \begin{equation}  \label{eq:choiceb}
 \b = \p_+ \oplus\p_-\oplus \a_1, \end{equation} so that 
$\n'_J =\b\oplus J\b$.

According to Proposition \ref{prop:main2}, the linear map   $\psi: \z_1\oplus \b\to \mathfrak{so}(\v)$  defined by  
\[
\psi(z)=j(z), \;z \in \z_1,\qquad\psi(z)= S(z)_{skew}, \; z\in \b , \]
satisfies $\text{Im}\, \psi \subset \u (n)$ and $\p_-\subset \ker S$, where $S$ was  defined in \eqref{S_z} and $S(z)_{skew}$ is given by \eqref{S^a}. We also have that
\[ j(z)= S(z)_{skew}, \; z\in \p _+, \qquad j(z)=j_-(z), \;  z\in \p _- .\]
 
 Consider the $2$-step nilpotent Lie algebra $\n_0=\b\oplus \v$, where $j_0(z)=j(z)$ for all $z\in \b$. It turns out that $\n$ is determined by $\left(\n_0,\pint , \v , J_\v ,\z_1, \psi \right)$  as follows: 
\begin{equation}\label{description}\begin{split} \ker (j)&=  J\z_1, \qquad\quad \quad \quad\;\;\; j(z)=\psi(z),\; z\in \z_1,  \\ j(z)&=j_0(z),\; z\in\b, \quad\quad j(w)=-\frac12\left[ J_\v ,j_0(Jw)\right]-\psi(Jw),\; \; w\in J\b.\end{split}
\end{equation}

This motivates the following definition.
\begin{definition}\label{def:tuple}
    Let  $\left(\n_0,\pint _0 ,  J_\v ,\z_1, \psi \right)$ be a 5-tuple where: 
\begin{enumerate}
\item[(i)] $\n_0$  is a $2$-step nilpotent or abelian Lie algebra  with an inner product $\pint _0$ such that $\b=[\n_0,\n_0 ]$ has even codimension in $\n_0$,
\item[(ii)] $\v$ is the orthogonal complement of $\b$ in $\n_0$,   $\dim \v  =2n$ and  $J_\v \in \mathfrak{so}(\v)$ is an arbitrary endomorphism satisfying  $J_\v^2=-I$, 
\item[(iii)]  $\z_1$ is an arbitrary vector space, which is non-zero when $\n_0$ is abelian, and \[\psi: \z_1\oplus \mathfrak b\to \mathfrak{u}(n)\] is a linear map such that $\psi|_{\z_1\oplus \p_+}\, $ is injective, $\p_-\subset \ker\psi$ and $\pi\left( \ker \psi \right)\subset \ker \psi$, where $\pi :\z_1\oplus\b\to \b$ is the natural projection,
\item[(iv)] the map $j_0:\b\to \s\o(2n)$ defined  in \eqref{jz}, which is injective since $\b=[\n_0,\n_0 ]$,  satisfies \[ \text{Im}\, \left( j_0\right)_+ \cap \, \text{Im}\, \psi =0 ,\]
 where  $\left( j_0\right)_+ (z)= j_0(z)_+$ for $ z\in \b$  (see \eqref{eq:T+-}).
Moreover, $\b$ admits a decomposition $\b=\p_+\oplus\p_-\oplus \a_1$ such that $j_0(z)\in \u (n)$ for $z\in \p_+$, $j_0(z)=\left(j_0 (z)\right)_-$ for $z\in \p_-$ and $\a_1$ is orthogonal to $\p_+\oplus\p_-$. 
    \end{enumerate}
    We call $\left(\n_0,\pint _0,  J_\v ,\z_1, \psi \right)$ a \textit{complex 2-step data}. We say that $\left(\n_0,\pint _0,  J_\v ,\z_1, \psi \right)$ is of type $(r,p_+, p_-, a_1,n)$ when 
    \[ r= \dim \z_ 1, \qquad p_\pm =\dim \p _\pm, \qquad a_1=\dim \a _1, \qquad 2n=\dim \v .  \] 
    \end{definition} 
    
    \smallskip

We show next that each complex 2-step data  $\left(\n_0,\pint _0,  J_\v ,\z_1, \psi \right)$ of type  $(r,p_+, p_-, a_1,n)$ gives rise to a 2-step nilpotent Lie algebra $\n$ with a 2-step complex structure $J$ such that 
\begin{equation*}
    \begin{split}
        \dim\n= 2(r +p_++p_-+a_1+n),& \qquad\quad \n '= \z _1\oplus \underbrace{\b\oplus J\b}_{\n '_J} , \\
 \dim \n'_J= 2(p_++p_-+a_1), \qquad &\dim \n'=r+\dim\n'_J , \qquad b_1= 2n+r, \end{split}
\end{equation*} where $b_1$ is the first Betti number of $\n$.
    
Let $\b=[\n_0 ,\n_0]$ and fix two real vector spaces $\widetilde{\b}$ and $\widetilde{\z_1}$   such that $\dim\widetilde{\b}=\dim \b$ and $\dim\widetilde{\z_1}=\dim \z_1$.
 Set \begin{equation}\label{eq:inverse-construction}
    \n :=\widetilde{\z_1}\oplus\z_1 \oplus \widetilde{\b}\oplus \n_0 ,\end{equation} and consider an  extension $J$ of $J_\v $ to $\n$ satisfying $J^2=-I$, $J\b=\widetilde{\b}$, $J\z_1=\widetilde{\z_1}$. Extend the inner product $\pint _0$ to $\pint$ on $\n$ in such a way that $\pint$ is Hermitian with respect to $J$ and the decomposition \eqref{eq:inverse-construction} is orthogonal. Let $j_0: \b \to \s\o(\v)$ be the injective linear map corresponding to the Lie bracket on $\n_0$.   We consider the  Lie bracket on $(\n , \pint )$   determined by the linear map $j:\widetilde{\z_1}\oplus\z_1 \oplus \widetilde{\b}\oplus \b \to \frak{so}(\v) $ defined as follows:
\begin{equation}\label{eq:description}\begin{split} j(z)&=0, \;\; z\in \widetilde{\z_1} , \qquad\quad\; j(z)=\psi(z),\; z\in \z_1,  \\ j(z)&=j_0(z),\; z\in\b, \quad\quad j(w)=-\frac12\left[ J_\v ,j_0(Jw)\right]-\psi(Jw),\; \; w\in \widetilde{\b}.\end{split}
\end{equation}
It follows that $\ker j = \widetilde{\z_1}$, that is, 
\[
    \n'= \z_1 \oplus \widetilde{\b}\oplus \b, 
\qquad\qquad  \n '_J= \widetilde{\b}\oplus \b . \]
We show next that  $J$ is an integrable 2-step complex structure on $\n$. 
 \begin{proposition}\label{prop:inverse-cons}
     Let $\left(\n_0,\pint _0 ,  J_\v ,\z_1, \psi \right)$ be a complex 2-step data  and let $\n$ be as in~\eqref{eq:inverse-construction}. Consider an extension $J$ of $J_\v$ to $\n$ and  extend $\pint _0 $ to $\pint$ on $\n$ as above. Then \eqref{eq:description} defines a Lie algebra structure on $\n$ such that $J $ is integrable and 2-step. Moreover,  
     \begin{equation}\label{eq:comm-2-step}
    \n'= \z_1 \oplus \widetilde{\b}\oplus \b, 
\qquad\qquad  \n '_J= \widetilde{\b}\oplus \b , \end{equation} and $(\n ,J)$ has no complex abelian factor.
 \end{proposition}
 \begin{proof}
  The integrability of $J$ follows from    Proposition \ref{prop:main2}.  
Since $ \n '\subset \z_1 \oplus \widetilde{\b}\oplus \b $, then $J\n '\subset \widetilde{\z_1} \oplus \z_1 \oplus \widetilde{\b}\oplus \b  \subset\z$, therefore,  $J\n '\subset \z$, that is, $J$ is 2-step (see Lemma \ref{J-2-pasos}).  
Note that conditions (iii) and (iv) in Definition \ref{def:tuple} and the fact that $j_0$ is injective, imply that $j: \z_1 \oplus \widetilde{\b}\oplus \b \to \s\o(\v )  $ is injective, hence \eqref{eq:comm-2-step} is satisfied. The fact that $(\n , J)$ has no  complex abelian factor follows from Lemma~\ref{lem:non-trivial-abelian}. 
\end{proof}

\begin{remark}\label{rem:Sz} We point out that the map $S: \z_0 \to  \operatorname{End}(\v)$ defined in \eqref{S_z} can be obtained from \eqref{eq:description} in terms of $j_0$ and $\psi$ as follows:\[ 
S(z)=\begin{cases} -\frac 12 \left\{ J_\v , j_0(z)\right\} +\psi (z), & \text{ if } z\in \b,\\
-\frac 12 J_\v\left\{ J_\v , j_0(Jz)\right\} +J_\v \circ \psi (Jz), & \text{ if } z\in J\b,\\
  -J_\v\circ \psi (z), & \text{ if } z\in \z_1,\\
  \psi (Jz),  & \text{ if } z\in J\z_1 ,
\end{cases} 
    \]
    where $\{A,B\} = A\circ B+B\circ A$ is the anticommutator of the operators $A$ and $B$.
\end{remark}

\smallskip

 Proposition \ref{prop:inverse-cons} gives that each complex 2-step data  $\left(\n_0,\pint _0 ,  J_\v ,\z_1, \psi \right)$ gives rise to a 2-step nilpotent Lie algebra equipped with a 2-step complex structure. Conversely, it follows from Proposition ~\ref{prop:main2}       that  all 2-step nilpotent Lie algebras with a  2-step complex structure are obtained  in this way for 
 an appropriate complex 2-step data   $\left(\n_0,\pint _0,  J_\v ,\z_1, \psi \right)$.

\begin{theorem}\label{main} 
Let $\n$ be a 2-step nilpotent Lie algebra with a 2-step complex structure~$J$ such that $(\n , J)$ has no  complex abelian factor. Then, for any   Hermitian inner product $\pint$ on $\n$, 
$(\n ,J , \pint )$ is obtained from a complex  2-step  data $\left(\n_0,\pint _0  ,  J_\v ,\z_1, \psi \right)$ as in Proposition \ref{prop:inverse-cons}
  where: 
\begin{enumerate}
    \item [$\ri$] $\v$ is the orthogonal complement of $\n'+J\n '$ in $\n$ and $J_\v=J|_\v$, 
  
\item [$\rii$] $\z_1$ is the orthogonal complement of $\n '_J$ in $\n '$ and 
  $ J\z_1 \perp \n '$, 
\item [$\riii$] $\b=\p_+\oplus\p_-\oplus \a_1 \subset \n '_J $, where 
\[\p_+\oplus J\p_+ =  \rlp\n_J'\rrp _+ , \qquad 
 \p_-\oplus J\p_-=\ker \left(S|_{\n'_J} \right), 
\]
and $\a=\a_1\oplus J\a_1$ is the orthogonal complement of $\rlp\n_J'\rrp _+\oplus\ker \left(S|_{\n'_J}\right) $ in $\n '_J$,
\item [$\riv$] $\n_0=\frak b \stackrel{\perp}{\oplus} \v$, $\pint _0= \pint |_{\n_0 \times \n_0}$ and the Lie bracket on $\n_0$ is given by $j_0(z)=j(z)$ for all $z\in \b$, hence, $\b$ is the commutator ideal of $\n_0 $, 
\item [$\rv$] $
\psi(z)=\begin{cases}
    j(z), \;& z \in \z_1,\\ S(z)_{skew}, & z\in \frak b , \;\; \text{$\rlp$see \eqref{S^a}$\rrp$}.\end{cases} $
\end{enumerate}
\end{theorem}


\smallskip

As a consequence of Theorem~\ref{main} and Corollary~\ref{abelian},  we obtain: 
\begin{corollary} Let $\n$ be a 2-step nilpotent Lie algebra with a complex structure $J$. Then 
\begin{enumerate}\item[$\ri$] $J$ is abelian if and only if  the subspaces $\p_-$ and $\a_1$ of $\n'_J$    are trivial.  
\item[$\rii$] $J$ is bi-invariant if and only if  the subspaces $\z_1$, $\p_+$ and $\a_1$ of $\n'$   are trivial.
\end{enumerate}
\end{corollary}

\smallskip 

\begin{example}
\label{ex:2-pasos} For each $n\in \N$, let $\n$ be a 2-step nilpotent $(2n+2)$-dimensional Lie algebra with a complex structure  $J$. Assume that $\n '=\z$ and $\dim \z =2$. It follows from Corollary \ref{cor:n'2} that $J$ is 2-step. Therefore, 
in view of Proposition \ref{prop:inverse-cons}, any complex structure on $\n$ can be constructed  from a complex 2-step data. Note that the first Betti number  of $\n$ is $b_1=2n$

Let  $\n_0=\R x \oplus\v$, $\dim\v=2n$, be a 2-step nilpotent Lie algebra with commutator $\R x$. It follows that $\n_0 $ is isomorphic to the Heisenberg Lie algebra $\h_{2n+1}$. 
Consider an inner product $\pint_0$ on $\n_0$  such that $x$ is a unit vector orthogonal to $\v$, and fix $J_\v\in\s\o(2n)$ such that $(\v, J_\v )$ is a complex vector space.    The Lie bracket on $\n_0$ is determined by the map $j_0(x)\neq 0$. Let $\psi: \R x\to \u(n)$ be a linear map such that $\R \left(j_0 (x)\right)_+ \cap \R \psi(x) =0$. If we set $\z_1=0$, then $\left(\n_0,\pint _0,  J_\v ,\z_1, \psi \right)$ is a complex 2-step data. Now  Proposition \ref{prop:inverse-cons} applies and we obtain a Lie algebra $\n= \R y \oplus \n_0$ with a complex structure $J$, where $J(x)=y$
and the inner product $\pint$ on $\n$ is the extension of $\pint _0$ such that $y$ is a unit vector orthogonal to $\n_0$. If we define the Lie bracket on $\n$  by
\[   j(x)=j_0(x), \qquad\quad j(y)= \frac12[J_\v , j_0(x)] +\psi(x),\]
then Proposition \ref{prop:inverse-cons} implies that $J$ is a 2-step complex structure on $\n$. We point out that
\begin{itemize}
    \item if $j_0(x)\in\u(n)$, then $j(y)=\psi(x)\in \u(n)$, therefore, $J$ is abelian,
    \item if $j_0(x)=\left(j_0(x)\right)_-$,  then $J$ is bi-invariant if and only if $\psi(x)=0$,
    \item if $j_0(x)\notin \u(n)$ and $j_0(x)\neq \left(j_0(x)\right)_-$, then $J$ is neither abelian nor bi-invariant.
     \end{itemize}
These family of Lie algebras exhausts the class of 2-step nilpotent Lie algebras $\n$  satisfying  $\dim \n ' =2$, $\z=\n '$ and such that $\n$ admits a complex structure. 
    \end{example}

\

\section{Three-step complex structures}\label{sec-3-step}
In this  section we give a characterization of the $2$-step nilpotent Lie algebras $\n$ admitting a 3-step complex structure $J$. It turns out that $(\n , J)$ can be obtained from  two proper subspaces  $\n_0, \, \n_1  \subsetneq\n $ each of which can be endowed with a Lie algebra structure such  that $J|_{\n_0} $ is 2-step and $J|_{\n_1}$ is abelian. We point out that $\n_0$ and $\n_1$ are not  Lie subalgebras  of $\n$. 

 If $J$ is 3-step  then $J\n ' \not \subset \z$ (see Lemma \ref{J-2-pasos}). Assume that $(\n ,J)$ has no complex abelian factor. Recall from Corollary \ref{cor:3-pasos} that the subspace $\n '_J= \n '\cap J\n '\neq 0$ and consider the following orthogonal decompositions:  
\begin{equation}\label{eq:decomp-3-step}    
 \n'= \z_1 \stackrel{\perp}{\oplus} \n '_J, \quad \z_1 \neq 0, \qquad J\z_1 = \z_2 \stackrel{\perp}{\oplus}\u, \quad \text{with} \; \z_2\subset \z, \;\; \u \cap \z =0. \end{equation}
Let $\v$ be the orthogonal complement of $\n '+J\n'$ in $\n$, hence $\v$ is $J$-invariant.  
Note that $\u  \neq 0$ since $J\n ' \not \subset \z$. Lemma \ref{idealJz} implies that $J\n '$ is an abelian ideal of $\n$. Since $\u$ is a complementary subspace of $(J\n')\cap \z$ in $J\n'$, then    $\u$ is abelian (see Lemma \ref{ideal}) and it follows from Corollary \ref{cor:3-pasos}  that $[\u ,\v ]\subset \n'_J$. 
 We define the following linear map:
\begin{equation}\label{eq:rho} \rho: \u \to \text{Hom}\,  (\v,\n '_J), \qquad \rho (u)v=[u,v], \; u\in \u, \, v\in \v. 
\end{equation}
For $v, w\in \v$, decompose 
\[ [v,w]=[v,w]_0 +[v,w]_1, \quad \text{ where } \;  [v,w]_0 \in\n ' _J , \;\;  [v,w]_1 \in\z_1.\] 
We define  
\begin{equation}\label{eq:mu}\begin{split} & \alpha  \in \Lambda^2(\v ^*)\otimes \n'_J, \qquad \alpha(v,w)= [v,w]_0, \\
& \mu \in \Lambda^2(\v ^*)\otimes \z_1, \qquad \;\mu(v,w)= [v,w]_1 .
\end{split}\end{equation}
Since $J$ is 3-step it follows that $\mu\neq 0$.  Set 
\begin{equation*}
    \n_0:= \n '_J \oplus \v, \qquad\qquad \qquad  
    \n_1:= \z_1 \oplus J\z_1\oplus \v, \quad   
 \end{equation*} 
 \[     J_\v=J|_\v, \qquad \qquad  J_0= J|_{\n'_J}    \qquad \qquad
    J_1=J|_{\z_1\oplus J\z_1}. \]
Consider the Lie brackets on $\n_0$ and $\n_1$ induced by $\alpha$ and $\mu$, respectively. In particular, $\alpha(\n_0,\n_0)\subset \n'_J \subset \z(\n_0)$ and $\mu(\n_1,\n_1)\subset \z_1 \subset \z(\n_1)$. 

We derive next the condition $N_J\equiv 0$ in terms of properties satisfied by $ \rho, \, \alpha$ and $ \mu$. 
\begin{proposition}\label{prop:caract-3-step}  Let $J$ be a 3-step complex structure on a 2-step nilpotent Lie algebra~$\n$. With the above notation, the following conditions are satisfied:  
\begin{enumerate}
\item[$\ri$] $\u$ is abelian, $\rho$ is injective and $\rho(u) \in \text{Hom}\, _\C (\v,\n '_J)$ for all $u\in \u$, that is, $J_0\,   \rho(u)v=\rho(u)J_\v  v$ for all $  u\in \u, \, v\in \v$,
\item[$\rii$] $J_0\oplus J_\v$ is a 2-step complex structure   on $(\n_0, \alpha)$, 
\item[$\riii$] $\mu\neq 0$ and $J_1\oplus J_\v$ is an abelian complex structure on $(\n_1 , \mu )$. 
\end{enumerate}
    
\end{proposition}
\begin{proof}
 As mentioned in the paragraph preceding \eqref{eq:rho}, $\u$ is abelian. The injectivity of  $\rho$ is equivalent to the fact that  $\u\cap \z =0$. Since $\u$ is abelian,      the condition $N_J\equiv 0$ is equivalent to: 
\begin{equation}\label{eq:integ-step3} N_J(u,v)=0,\quad u\in \u, \, v\in \v, \quad  \text{ and } \quad 
 N_J(v,w)=0, \quad v,w\in \v .
\end{equation}
If $u\in \u$ and $v\in \v$, $N_J(u,v)=0$ amounts to $[u,v]=-J[u,Jv]$, which is equivalent to 
\begin{equation}\label{eq:Hom_C} J_0\,\rho(u)v=\rho(u)J_\v v, \qquad u\in \u, \, v\in \v.
\end{equation}
The above equation says that $\rho(u) \in \text{Hom}\, _\C (\v,\n '_J)$ for all $u\in \u$, that is, condition (i) in the statement is satisfied.

 The remaining conditions will be imposed by setting  
 $N_J(v,w)=0$ for $v,w\in \v$. We compute
\begin{equation}\begin{split} N_J(v,w)=&[v,w]+J([Jv,w]+[v,Jw])-[Jv,Jw]\\
&= [v,w]_0+[v,w]_1+J[Jv,w]_0 +J[Jv,w]_1 \\ 
& \; +J[v,Jw]_0 +J[v,Jw]_1-[Jv,Jw]_0-[Jv,Jw]_1.  
\end{split}
\end{equation}
Since $N_J(v,w)\in \n'_J\oplus \z_1\oplus J\z_1$, it follows  that $N_J(v,w)=0$ for $ v, w\in \v$ if and only if each component of $N_J(v,w)$ vanishes, that is,  \begin{eqnarray*} & &[v,w]_0+J[Jv,w]_0+J[v,Jw]_0-[Jv,Jw]_0=0, \\
& & [v,w]_1-[Jv,Jw]_1=0, \qquad J[Jv,w]_1 +J[v,Jw]_1=0.
\end{eqnarray*}
The above equations are equivalent to:
\begin{eqnarray}   
& &J_0\oplus J_\v \text{ is integrable on } (\n_0, \alpha ),\\ \label{eq:J-1-abeliana} & & \mu(J v,J w)=\mu(v,w), \quad v, w\in \v .
\end{eqnarray} 
Note that $J_0\oplus J_\v$ is 2-step since $J_0\, \alpha(v,w)\subset J_0\, \n_J '=\n'_J\subset \z(\n_0)$. Therefore, (ii) follows.

Equation \eqref{eq:J-1-abeliana} means that $J_1\oplus J_\v$ is abelian on $(\n_1 ,\mu )$. 
We point out that $\mu=0$ would imply that $\n '= \n '_J$, which is impossible in view of \eqref{eq:decomp-3-step}. This completes the proof of (iii).

\end{proof}

\begin{remark}
    We point out that the Lie algebra $(\n_0 ,\alpha )$ defined above can be abelian. In this case, $[\v ,\v]=\z_1$ and $[\u , \v]=\n '_J$. 
    
    On the other hand, Proposition \ref{prop:caract-3-step} (iii) gives that $(\n_1, \mu)$ is never an abelian Lie algebra.
\end{remark}

\begin{lemma} \label{lem:twisted-sum} Let $(\n_0 ,\alpha)$ and $(\g , \mu)$ be 2-step nilpotent Lie algebras where
\[ \n_0= \q\oplus \v, \qquad  \g=\z_1 \oplus \v,\]
   with $\alpha\in \Lambda^2(\v ^*)\otimes \q$ and $\mu\in \Lambda^2(\v ^*)\otimes \z_1$. Fix an arbitrary non-zero vector space $\u$ and an injective linear map $\rho: \u \to \text{Hom}\,  (\v,\q )$. Set $\n:= \u\oplus\z_1\oplus \q\oplus \v$ and define a  Lie bracket $\nu$ on $\n$ as follows:
   \[ \nu |_{\Lambda^2 (\v )}=\alpha\oplus \mu, \qquad \nu(u,v)=-\nu(v,u)=\rho(u)v, \quad u\in \u, \, v\in \v .\]
   Let $\n'=\nu(\n , \n )$ and $\z$ denote the commutator and center of $(\n , \nu )$, respectively. Then $\u\cap \z=0$. Moreover,     $\n '=\z_1\oplus \q $ if and only if the following conditions are satisfied:
   \begin{enumerate}
       \item [$\ri$] For every $z\in \z_1$ there exist $y\in \Lambda ^2(\v )$ and $u\in \u$ such that $\begin{cases}
           z=\mu(y),\\ \alpha(y)\in \text{Im}\, \rho(u),
       \end{cases}$
       \item [$\rii$] for every $x\in \q$ there exist $y\in \ker\mu \subset \Lambda ^2(\v )$ and $u\in \u$ such that $x-\alpha (y) \in \text{Im}\, \rho(u).$
   \end{enumerate}
\end{lemma}
\begin{proof} The fact that $\u\cap \z=0$ follows from the injectivity of $\rho$. 

 The commutator ideal $\n'=\nu(\n , \n )$ is given by
\[ \n'= (\alpha \oplus \mu)(\Lambda ^2(\v )) + \rho(\u)\v \subset \z_1\oplus \q \subset \z. \]
Assume that $\n'=\z_1\oplus \q $, then $\z_1\subset \n '$ and $\q\subset \n'$. Since $\z_1\subset \n '$, for each  $z\in \z_1$ there exist $y\in \Lambda ^2(\v )$, $u\in \u$, $v\in \v$ such that $z=\alpha(y)+ \mu(y)+\rho(u)v$. Therefore, $z=\mu(y)$ and $\alpha(y)+ \rho(u)v=0$, that is, $\alpha(y)=- \rho(u)v \in \text{Im}\, \rho(u)$, and (i) is satisfied. On the other hand, since $\q\subset \n'$,  it follows  that for every $x\in \q$ there exist $y\in \Lambda ^2(\v )$, $u\in \u$, $v\in \v$ such that $x=\alpha(y)+ \mu(y)+\rho(u)v$. Therefore, $\mu(y)=0$ and $x=\alpha(y)+\rho(u)v$, that is, $y\in \ker\mu$ and  $x-\alpha (y) = \rho(u)v \in \text{Im}\, \rho(u),$ hence, (ii) is satisfied.

Conversely, it is easily checked that conditions  (i) and (ii) imply that $\z_1\subset \n '$ and  $\q\subset \n '$, respectively. Therefore, $\z_1\oplus\q\subset \n '$ and the lemma follows.
    
\end{proof}

\ 

Given three  even dimensional vector spaces  $\v$, $\q$ and $\h$, fix $J_\v\in \text{End}\,(\v )$, $J_0\in \text{End}\,(\q )$ and $J_1\in \text{End}\,(\h )$ such that $J_\v^2=-I,\; J_0^2=-I, \; J_1^2=-I$. Consider a decomposition
\[ \h= J_1\z_1 \oplus \z_1, \quad \text{ with } J_1\z_1 =\z_2\oplus \u, \;\; \u \neq 0.\]
Set 
\[ \n_0= \q\oplus \v, \qquad  \n_1=\h\oplus \v,\]
and let $\alpha\in \Lambda^2(\v ^*)\otimes \q$ and $\mu\in \Lambda^2(\v ^*)\otimes \z_1$. We obtain two 2-step nilpotent Lie algebras $(\n_0,\alpha)$ and $(\n_1,\mu)$. Fix a linear map $\rho: \u\to \text{Hom}\,(\v , \q )$ and  set 
\[\n := \h\oplus\q\oplus\v, \qquad\qquad J:=J_1\oplus J_0\oplus J_\v.\]
Define a Lie bracket $\nu$ on $\n$ as follows:
\begin{equation}\label{eq:nu}
    \nu |_{\Lambda^2 (\v )}=\alpha\oplus \mu, \qquad \nu(u,v)=-\nu(v,u)=\rho(u)v, \quad u\in \u, \, v\in \v .\end{equation}
The next theorem gives necessary and sufficient conditions  for $J$ to be a 3-step complex structure on $(\n , \nu )$. 
\begin{theorem} \label{teo:main-3-step}
With the above notation, $J$ is  a 3-step complex structure on $(\n , \nu )$ if and only if the following conditions are satisfied:
\begin{enumerate}
    \item[$\ri$] $\u$ is abelian, $\rho$ is injective and $\rho(u) \in \text{Hom}\, _\C (\v,\q)$ for all $u\in \u$, that is, $J_0\,   \rho(u)v=\rho(u)J_\v  v$ for all $  u\in \u, \, v\in \v$,
\item[$\rii$] $J_0\oplus J_\v$ is a 2-step complex structure   on $(\n_0, \alpha)$, 
\item[$\riii$] $\mu\neq 0$ and $J_1\oplus J_\v$ is an abelian complex structure on $(\n_1 , \mu )$. 
\end{enumerate}
    Moreover, $\n '=\z_1\oplus \q$ and $\n'_J=\q$ if and only if conditions $\ri$ and $\rii$ of  Lemma \ref{lem:twisted-sum} are satisfied. 
\end{theorem}

\begin{corollary} Any 3-step complex structure $J$ on a 2-step nilpotent Lie algebra $\n$ is obtained as in Theorem \ref{teo:main-3-step} from  $(\n_0 , \alpha)$ and $(\n _1 , \mu )$, where $\q =\n'_J$ and $\n'= \z_1\oplus\q$. 
    
\end{corollary}

\begin{example}\label{ex:3-step}
     For each $n\geq 3$, we construct next the  class  of all $2n$-dimensional 2-step nilpotent Lie algebras $\n$ admitting a complex structure $J$ such that  $\n '=\z$ and $\dim \z =3$. It follows from Corollary \ref{cor:odd}  that these Lie algebras only admit 3-step complex structures.

     Let $\v$ be a $(2n-4)$-dimensional vector space and consider the endomorphism $J_\v$ which is represented by the following matrix in the standard basis $\{e_1,\ldots ,e_{2n-4}\}$ of $\v$: \[
     J_\v=\begin{pmatrix} 0_{n-2}&-I_{n-2} \\
     I_{n-2} & 0_{n-2}
         \end{pmatrix},\]
where $I_{n-2}$ and $0_{n-2}$ are the identity and zero $({n-2})\times ({n-2})$ matrices, respectively. Set $\n_0=\q\oplus \v$, where $\q=\text{span}\, \{f_1,f_2\}$, and $\n _1=\R y\oplus\R x\oplus \v $. Fix a Lie bracket $\al$ on $\n _0$ from the family constructed in Example~\ref{ex:2-pasos} and let $J_0f_1=f_2$. Consider the following Lie bracket $\mu$ on $\n_1$:
\[ \mu(e_{k}, e_{k+n-2})=x , \qquad k=1, \ldots , n-2 ,
\]
and define $J_1x=y$. Consider $\rho(y):\v\to \q$ such that the matrix of $\rho(y)$ is given as follows with respect to the bases $\{e_1,\ldots ,e_{2n-4}\}$ and $\{f_1 , f_2\}$:
\[ \rho(y)=\begin{pmatrix} a_1& a_2& \cdots & a_{n-2}&-b_1& -b_2& \cdots & -b_{n-2}\\
b_1&b_2&  \cdots & b_{n-2} &     a_1&a_2&\cdots & a_{n-2}
\end{pmatrix}, \qquad a_i, \, b_i\in \R, 
\]
which satisfies $\rho(y) J_\v= \begin{pmatrix}
    0&-1\\ 1&0
\end{pmatrix} \rho(y)$. We assume that some of the scalars $a_i$ or $b_i$ is no-zero, so that rank$\, \rho(y)=2$. Note that $\n_1$ is isomorphic to $\R y\oplus \h_{2n-3}$ and  $J_1\oplus J_\v $ is an abelian complex structure on $\n_1$. In this case, $\n_1$ has $\left[ \frac n2\right]$ complex structures up to equivalence (see \cite[Proposition 2.2]{ABD}). 
Set $\n= \q\oplus\n_1$, $J=J_0\oplus J_1\oplus J_\v$ and let $\nu$ be the Lie bracket on $\n $ as defined in \eqref{eq:nu}. In this case, we have that $\u=\R y$ and $\z_1 =\R  x$. Since $\rho, \, (\n_0, \al, J_0\oplus J_\v )$ and $  (\n_1, \mu, J_1\oplus J_\v )$ satisfy conditions (i), (ii) and (iii) of Theorem \ref{teo:main-3-step}, we have that $J=J_0\oplus J_1\oplus J_\v$ is a complex structure on $(\n, \nu )$, which is 3-step since $\mu\neq 0$. Since rank$\, \rho(y)=2$, it follows that $\n '=\q\oplus \R x$ and $\n '_J=\q$.

We point out that for $n=3$, the 6-dimensional Lie algebra constructed in this way is isomorphic to the free 2-step nilpotent Lie algebra $\f_3$ of rank $3$.  
\end{example}

\ 

\section{Pluriclosed metrics}

{We recall from \cite{AN} the following important  results.}


\begin{lemma}\cite[Lemma 3.4]{AN}\label{lem:SKT-center-J-inv}
     Let $\pint$ be a pluriclosed metric on $(\n, J )$, where  $\n$ is 2-step nilpotent. Then,  the center $\z$ of $\n $ is given by:
     \[ \z= \{y\in \n : [y,Jy]=0\}.\]
\end{lemma}

\begin{theorem}\cite[Theorem 1.2]{AN}\label{thm:SKT-2-step}
   Let $\n$ be a nilpotent  Lie algebra and let $J$ be a complex structure on $\n$. If $(\n ,J)$ admits a pluriclosed metric,  then $\n$ is 2-step nilpotent or abelian.
\end{theorem}
\begin{remark}    
It was  proved in \cite[Proposition 3.1]{EFV} that if $(\n , J, \pint )$ is a nilpotent  Hermitian Lie algebra such that $\pint$ is pluriclosed, then the center $\z$ of $\n$ is $J$-invariant. This fact 
also follows from Lemma \ref{lem:SKT-center-J-inv} and Theorem \ref{thm:SKT-2-step}.  \end{remark}
We prove next that a Lie algebra $\g$ with a bi-invariant complex structure admits no compatible pluriclosed metric unless $\g$ is abelian.  This fact was already known for nilpotent
Lie algebras of dimension $\geq 6$  having a rational form \cite[ Proposition 23]{Ug}. 

\begin{corollary} Let $\n$ be a nilpotent Lie algebra with a bi-invariant complex structure $J$. If there exists a pluriclosed inner product $\pint$
    on $(\n, J)$, then $\n$ is abelian. 
\end{corollary}
\begin{proof}
    Assume that there exists  a pluriclosed inner product $\pint$
    on $(\n, J)$, then it follows from Theorem \ref{thm:SKT-2-step} that $\n$ is 2-step nilpotent or abelian. If $\n$ were not abelian, then it would be 2-step nilpotent and Lemma \ref{lem:SKT-center-J-inv} would imply that $\z= \{y\in \n : [y,Jy]=0\}.$ Since $J$ is bi-invariant, then $[y,Jy]=J[y,y]=0$ for all $y\in \n$, that is, $\z=\n$, hence, $\n$ is abelian, a contradiction. 
\end{proof}

\begin{corollary}\label{cor:3-step-no-skt}
 If $\n$ is nilpotent and $J$ is a 3-step complex structure on $\n$, then no Hermitian metric on $(\n ,J)$    is  pluriclosed.
 \end{corollary}
\begin{proof}
 Since $J$ is 3-step, then $\n$ is not abelian. If $\pint$ were a pluriclosed metric on  $(\n ,J)$, then $\n$ would be  2-step nilpotent (Theorem \ref{thm:SKT-2-step}), hence  $\n '\subset \z$.   As $J$ is  3-step,   there exists $y \in  \n '$ such that $Jy \notin \z$,  therefore, $[y, Jy]\neq 0$ (Lemma \ref{lem:SKT-center-J-inv}). On the other hand, $y\in \n '\subset\z$, therefore, $[y, Jy]=0$, a contradiction.
\end{proof}

The next corollary is a straightforward consequence of Theorem \ref{thm:SKT-2-step} and Corollary  \ref{cor:3-step-no-skt}.
\begin{corollary}\label{cor:skt-J-2step}
    Let $\n$ be a nilpotent Lie algebra with a complex structure $J$. If  $(\n ,J)$    admits a  pluriclosed metric, then $\n$ is abelian or  $J$ is 2-step.
\end{corollary}

\begin{theorem}    
\label{prop:skt-2-step}
  Let   $(\n ,J, \pint )$ be a Hermitian Lie algebra such that $\n$ is  2-step nilpotent. Then $\pint$ is pluriclosed if and only if $J$ is 2-step and the following condition is satisfied:
  \begin{equation}\label{eq:skt-J-2step}   
  \begin{split}    &\la [Jy, J z], [w,u]\ra    
   - \la [Ju, J z], [w, y ]\ra   
   + \la [Ju, Jy], [w, z]\ra     \\
 & + \la [Jw, J z], [u, y ]\ra    
   - \la [Jw, Jy], [u, z]\ra  
   + \la [Jw, Ju], [y , z]\ra =0,   
\end{split}
\end{equation}
for all $w,u,y,z\in \n$.

In particular, if $J$ is abelian, then $\pint$  is pluriclosed if and only if 
\begin{equation}\label{skt}
j([u,y])z + j([y,z])u+j([z,u])y = 0, \quad \text{ for all } u,y,z\in \v ,
\end{equation}
where  $\v =\z ^\perp$ and $j:\z \to \operatorname{End}\, ( \v)$ is defined in \eqref{jz} by taking $\z_0=\z$, the center of $\n$.
\end{theorem}
\begin{proof}
Assume that $\pint$ is pluriclosed, then Corollary \ref{cor:skt-J-2step} implies that  $J$ is 2-step. Moreover, the condition $dc=0$ is equivalent to \eqref{eq:skt-J-2step} since $J\n '\subset \z$ (see \eqref{eq:dc}). 

Conversely, assume that $J$ is 2-step and \eqref{eq:skt-J-2step} is satisfied. Since $J\n '\subset \z$, then $dc$ is given by 
\begin{equation*}   
  \begin{split}    dc(w,u,y,z)=&\la [Jy, J z], [w,u]\ra    
   - \la [Ju, J z], [w, y ]\ra   
   + \la [Ju, Jy], [w, z]\ra     \\
 & + \la [Jw, J z], [u, y ]\ra    
   - \la [Jw, Jy], [u, z]\ra  
   + \la [Jw, Ju], [y , z]\ra .  
\end{split}
\end{equation*}
Since, by assumption, \eqref{eq:skt-J-2step} is satisfied, then $dc=0$, that is, $\pint$ is pluriclosed.

    When $J$ is abelian,  the pluriclosed condition \eqref{eq:skt-J-2step} becomes:
\begin{equation}\label{eq:skt-J-abelian}
        2\left( \la [y,  z], [w,u]\ra    
   - \la [u,  z], [w, y ]\ra   
   + \la [u, y], [w, z]\ra     
       \right) =0, 
  \end{equation}
  for all $w,u,y,z \in \n$ (compare with \cite[Proposition 2.2]{FTV}). Equation \eqref{eq:skt-J-abelian} is clearly satisfied when one of $w,u,y,z $ belongs to $\z$. Therefore, the pluriclosed condition is equivalent to \begin{equation}\label{eq:skt-abelian}      
   \la [y,  z], [w,u]\ra    
   - \la [u,  z], [w, y ]\ra   
   + \la [u, y], [w, z]\ra =0, \quad \text{ for all } w,u,y,z \in \v ,
   \end{equation}
where $\v = \z ^\perp $. Using the definition of $j: \z \to \operatorname{End}\, ( \v)$ given in \eqref{jz}, with $\z_0=\z$, equation \eqref{eq:skt-abelian} becomes:
\begin{eqnarray*} 0&=&\la j([y,z])w,u \ra- \la j([u,z])w,y \ra + \la j([u,y])w,z \ra \\
& =& -\la w,  j([y,z])u \ra+ \la w, j([u,z])y \ra - \la w, j([u,y])z \ra \\ & =& -\la w,  j([y,z])u + j([z,u])y + j([u,y])z \ra    ,                      \end{eqnarray*}
for all $u,y,z,w\in \v$, which is equivalent to:
\[ j([y,z])u +  j([z,u])y + j([u,y])z =0,
\]
for all $u,y,z\in \v$, and the theorem follows.
\end{proof}

\

In the next theorem, $\h _{2m+1}$ denotes the $(2m+1)$-dimensional Heisenberg Lie algebra, which has a basis $\{ x_1, \ldots, x_m,  y_1, \ldots , y_m, z\}$  such that the non-trivial Lie brackets are given by $[x_i,y_i]=-[y_i, x_i]=z$ for all $i$.  
\begin{theorem}
    \label{prop:skt-nilp} Let $(\n , J, \pint )$ be a Hermitian  nilpotent   Lie algebra such that $\dim \n ' =1$. Then $ \pint $ is pluriclosed if and only if $\n$ is isomorphic to $\R^{2d-3}\oplus  \h_3$, where   $\dim\n =2d$. 
\end{theorem}
\begin{proof} It follows from 
\cite[Theorem 4.1]{BD} that a nilpotent Lie algebra $\n$ such that $\dim \n '=1$ is $2$-step nilpotent and isomorphic to $ \R^s \oplus \h _{2m+1}$. If $J$ is a complex structure on $\n$, then it was shown in 
\cite[Proposition 2.2]{Rol} that $J$ is abelian. 

Assume that $\pint $ is pluriclosed. We wish to show that $m=1$. Let $z\in \n'$ be a unit vector, and  decompose $\n$ orthogonally as $\n=\z\oplus \v$.   Since  $\pint$ is Hermitian and $J$ is abelian, then both, $\z $ and $\v$, are $J$-invariant. It follows from Theorem  \ref{prop:skt-2-step} that the pluriclosed condition in this case is equivalent to \eqref{skt}. 

  Let $\mu\in \Lambda^2(\n ^*)$   be defined by $[x,y]=\mu(x,y)z$ for all $x,y \in \n$. 
 Since $\n$ is isomorphic to $ \R^s \oplus \h _{2m+1}$, it follows that $2m=\text{rank}\, \mu$, therefore, there exist linearly independent $x_1, \ldots , x_m$, $y_1,\ldots , y_m   \in \v$ generating a subspace $\k \subset \v $ such that the matrix of the restriction of $\mu$ to $\k$ in this basis is given by:
\[ \mu |_{\k\times\k } =\begin{pmatrix} 0 & -I \\
I & 0
\end{pmatrix},
\]
where $I$ is the identity $(m\times m)$-matrix. In other words, $[x_i, y_i] =- [y_i, x_i]=z$ for all $i$, hence, $\k\oplus \R z= \h_{2m+1}$. Assume that $m>1$ and compute \eqref{skt} for $x_1,y_1,x_2$:
\[ j([x_1,y_1]) x_2+ j([y_1,x_2])x_1+j([x_2,x_1])y_1 = j(z)x_2 \neq 0,
\]
since $\la j(z)x_2 , y_2\ra = \la z, [x_2, y_2]\ra =\la z, z \ra=1$. Therefore, $ \pint $ is not pluriclosed when $m>1$. Clearly, \eqref{skt} is satisfied for $m=1$, and the theorem follows.
\end{proof}

\begin{remark}
    According to Theorem \ref{prop:skt-nilp},  the Lie algebras  $\n=\R^{2k+1} \oplus \h _{2m+1}$ with $k\geq 0$, $m\geq 2$, 
are examples of 2-step nilpotent Lie algebras such  that no complex structure on $\n$ admits compatible pluriclosed metrics. \end{remark}

\begin{remark}
    It follows from  Theorem \ref{prop:skt-nilp} that no Hermitian metric on $\R^{2k+1}\oplus  \h_{2m+1}$, $m\geq 2$, is pluriclosed.   On the other hand, it was shown in \cite{AV} that $\R^{2k+1}\oplus  \h_{2m+1}$ admits a balanced Hermitian structure if and only if $m\geq 2$. 
\end{remark}


\smallskip

\subsection{Pluriclosed metrics compatible with a class of abelian complex structures.} In \cite{Tam} a construction of  $2$-step nilpotent metric Lie algebras $(\n , \pint )$ was given from
homogeneous fiber bundles over compact irreducible symmetric spaces. We will consider the following particular case of this construction. 
Let $(\g ,\h )$ be a compact semisimple symmetric pair of Lie algebras, that is, $\g$ is a compact semisimple Lie algebra, $\h$ is a compact subalgebra of $\g$  and if $\m$ denotes the orthogonal complement of $\h$ in $\g$ with respect to the Killing form $B$ of $\g$, then $\m$ satisfies: 
\[  [\h,\m]\subset \m , \qquad\quad [\m ,\m ]\subset \h .  
\] 
Set $\n=\n(\g,\h):=\h \oplus \m$ (i.e., $\n$ and $\g$ are isomorphic as vector spaces) with the inner product $\pint =-B$ and the following Lie bracket $[\, \cdot  ,\, \cdot ]_\n$:
\begin{equation}\label{eq:sspair} [z+x , z'+x' ]_\n  = [x,x'], \qquad z,z' \in \h , \; x,x'\in \m ,   \end{equation}
where $[ \,\cdot , \, \cdot ]$ is the Lie bracket on $\g$. 
Then $(\n , [\,\cdot ,\,\cdot ]_\n)$ is a $2$-step nilpotent  Lie algebra (see \cite{Tam}).

Recall that $(\g ,\h )$ is said to be irreducible  when the adjoint representation of $\h$ on $\m$ is irreducible.  
In this particular case,  Proposition 3.3 in \cite{Tam} implies the following result. 
\begin{proposition}[{\cite[Proposition 3.3]{Tam}}] \label{prop:Tam}
Let  $(\n (\g ,\h ),\pint)$ be the  $2$-step nilpotent  Lie algebra corresponding to the compact semisimple symmetric pair $(\g ,\h )$ as in \eqref{eq:sspair}. 
Assume that $(\g ,\h )$ is irreducible. Then:
\begin{enumerate}
\item[$\ri$] the center of $\n (\g ,\h )$ coincides with $\h$, 
\item[$\rii$] $j(z)=\ad_z |_\m$, for $z\in \h$, where $\ad$ is the adjoint representation of $\g$. 
\end{enumerate}    
\end{proposition}

Let $(\g ,\h, J, -B )$ be a compact semisimple  irreducible  Hermitian symmetric pair, that is, $(\g ,\h )$ is a  compact semisimple  irreducible   symmetric pair and $J:\m \to \m$ is an orthogonal endomorphism such that $J^2=-I$ and whose extension to $\g$ by $J|_\h =0$ satisfies:
\[ J[x,y]=[x,Jy], \quad \text{ for } x\in\h, \; y\in \m , \qquad N_J(x,y)\in \h , \quad \text{ for } x, y\in\g .  \]
  If $\dim \h \equiv s \pmod{2}$, $s=0$ or $1$, define $\J:\R^s\oplus\g \to \R^s\oplus\g$ as follows:
\begin{itemize}
\item If $\dim \h$ is even ($s=0$) take an  {orthogonal} endomorphism $J_1$ of  $\h$ such that {$J_1=-I$} and define $\J: \g \to  \g$ by  $\J|_\h= J_1, \;  \J|_\m= J$.
\item If $\dim \h$ is odd ($s=1$),  consider $\R\oplus\g$ and extend $-B$ to an inner product on $\R\oplus\g$ such that $\R $ is orthogonal to $\g$.   
Let $J_1$ be an {orthogonal} endomorphism of  $\R\oplus\h$ such that $J_1^2=-I$ and define  $\J:\R\oplus \g \to \R\oplus \g$ by 
$\J|_{\R\oplus\h}= J_1, \; 
 \J|_\m= J$.
\end{itemize}
             
\begin{proposition}\label{nat-red-skt} Let $\n (\g ,\h)$ be the $2$-step nilpotent Lie algebra constructed from a compact semisimple  irreducible  Hermitian symmetric pair $(\g ,\h, J, -B)$ as in \cite{Tam}. Let $s=0$ or $1$, where  $\dim \h \equiv s \pmod{2}$. Then the  $2$-step nilpotent Lie algebra   
 $(\R ^s \oplus \n (\g ,\h) ,  \pint )$, with $\pint =-B$,  admits  
 an  abelian complex structure such that  $ \pint $ is  pluriclosed.  Moreover, if $N(\g ,\h)$ is the simply connected Lie group with Lie algebra $\n (\g ,\h)$ and $ \pint$ is the corresponding left invariant metric on $\R^s\times N(\g ,\h)$, then $\pint$ is naturally reductive.
\end{proposition}
\begin{proof}  
Let $\n=\R ^s \oplus \n (\g ,\h)$ with the Lie bracket defined in \eqref{eq:sspair}.  
It follows from the proof 
 of \cite[Theorem 9.6 (1), Ch. XI]{KN}) that $J$ satisfies:  
\begin{equation}\label{Hermit-symm}   [Jx,Jy]=[x,y], \qquad \text{ for all } x, y \in \m ,
\end{equation}
where $[\cdot , \cdot ]$ is the Lie bracket on $\g$. 
Consider any orthogonal endomorphism $\J$ as defined above the statement, hence $\pint$ is Hermitian. Equation \eqref{Hermit-symm}  
says that $[\J x,\J y]_\n =[x,y]_\n $ for $x,y\in \m$, and since the center of $\n$  is $\J$-invariant, it follows that $[\J x,\J y]_\n =[x,y]_\n $ for all $x,y\in \n$, that is, $\J$ is abelian on $\n$.  Theorem  \ref{prop:skt-2-step} implies that  $\pint$ is pluriclosed if and only if \eqref{skt} is satisfied for all $x,y,z \in \m$. Proposition \ref{prop:Tam} (ii)  gives that \eqref{skt} is equivalent to the Jacobi identity on $\g$ for $x,y,z \in\m$.

For the last assertion, see Remark \ref{rem:skt-nat-red} in \S\ref{sec-nat-red}.
\end{proof}

\

\section{Abelian complex structures on  naturally reductive \\ homogeneous nilmanifolds} \label{sec-nat-red}
We start this section by recalling from \cite{Lau} the description of naturally reductive homogeneous nilmanifolds via representations.

Let $N$ be a simply connected real nilpotent Lie group  endowed with a left invariant Riemannian metric, denoted by $(N ,\pint)$, where $\pint$ is the inner product on 
the Lie algebra $\n$ of $N$ determined by the metric. The Riemannian manifold $(N ,\pint)$ 
is said to be a (simply connected) homogeneous nilmanifold. When $N$ is a $2$-step nilpotent Lie group then $(N ,\pint)$ is called 
 a (simply connected) $2$-step homogeneous
nilmanifold. 

Natural reductivity on homogeneous nilmanifolds has been studied by C. Gordon in \cite{G} (see also \cite{K,TV}). It is proved in \cite{G}  that if $(N ,\pint)$ is a naturally reductive homogeneous nilmanifold then $N$ must be at most $2$-step nilpotent and a characterization in the $2$-step nilpotent  case is given. Later in \cite{Lau} J. Lauret gave a description of naturally reductive $2$-step homogeneous Riemannian nilmanifolds using representations of compact Lie algebras as follows. Assume that:  
\begin{enumerate}
\item[$\ri$] $\h$ is a compact Lie algebra with an ad$\,\h$-invariant inner product $\pint_\h$,
\item[$\rii$] $(\pi , V)$ is a faithful real representation of $\h$ without trivial subrepresentations, that is,   $\ker \pi =0$ and 
$\bigcap _{x\in \h} \ker\pi(x)=0$. Let $\pint_V$ be a $\pi(\h)$-invariant inner product on $V$.
\end{enumerate}
Consider $\n=\h\oplus V$ with the following Lie bracket:
\begin{eqnarray}
 [\h,\h]_\n&=&[\h,V]_\n=0, \qquad \qquad [V,V]_\n\subset \h,\\ \label{eq:pi_z}
  \la [v,w]_\n, x\ra _\h & =&\la \pi(x)v, w  \ra _V, \qquad \quad \text{for all }x\in \h, \; v,w\in V,
\end{eqnarray}
and endow $\n$ with the inner product $\pint$ such that:
\[ \pint|_{\h\times\h}=\pint_\h, \qquad \pint|_{V\times V}=\pint_V, \qquad \la \h , V\ra=0.
\]
Condition $\rii$ above ensures that $\h$ is precisely the center of $\n$ and it follows from \eqref{eq:pi_z} that the maps $j(x), \; x\in\h $ (see \eqref{jz}),   are given by $j(x)=\pi(x)$.

Let $N$ be the simply connected Lie group with Lie algebra $\n$ and 
endow $N$ with the left invariant metric determined by $\pint$, obtaining a $2$-step
homogeneous nilmanifold denoted by $(N(\h, V), \pint)$, which is naturally reductive with respect to its full isometry group I$(N(\h, V), \pint) $,   and $(N(\h, V), \pint)$ has no   Euclidean factor since  $\ker\pi =0$. Conversely,  any homogeneous nilmanifold $(N, \pint)$ without
Euclidean factor which is naturally reductive with respect to I$(N(\h, V), \pint) $  can be constructed in this way:
\begin{theorem}[{\cite[Theorem 2.7]{Lau}}] The $2$-step
homogeneous nilmanifold $(N(\h, V), \pint)$ constructed as above has no  Euclidean factor and 
is  naturally reductive with respect to ${\rm I}(N(\h, V), \pint)$. Moreover, any homogeneous nilmanifold $(N, \pint)$ without
Euclidean factor which is naturally reductive with respect to ${\rm I}(N, \pint)$ can be constructed in this way. 
\end{theorem} 

\begin{remark}\label{rem:skt-nat-red} Let $ (\g ,\h )$ be a  compact semisimple irreducible symmetric pair and consider the metric Lie algebra $(\n (\g ,\h ),\pint)$ constructed  as in \eqref{eq:sspair}. It follows from Proposition \ref{prop:Tam}  that the corresponding simply connected $2$-step nilmanifold $(N(\h, V), \pint)$ is     contained in the class of naturally reductive $2$-step nilmanifolds  by taking $V=\m$ and  $\pi(z)=\ad_z |_\m$, for $z\in \h$, where $\ad$ is the adjoint representation of $\g$. Therefore, the pluriclosed metric from Proposition \ref{nat-red-skt} is naturally reductive.
\end{remark}

Given a compact Lie algebra $\h$, $(\pi, V)$ a  real representation of $\h$ and a $\pi(\h)$-invariant inner product $\pint$ on $V$, decompose $V$ into an orthogonal direct sum of isotypic components: 
\begin{equation}\label{isotypic} V=V_1^{\oplus r_1}\oplus \cdots \oplus V_k^{\oplus r_k}, \quad V_i \text{ irreducible } V_i\not \simeq V_j, \; i\neq j,
\end{equation}
where   $V_i^{\oplus r_i}=V_i \oplus \cdots \oplus V_i $ ($r_i$ copies). The positive integer $r_i$ is called the multiplicity of $V_i$. Since each $V_i$ is a real irreducible representation, we have that 
\[ \text{End}\,_\h (V_i)=\{T\in \text{End}\,(V_i): T\circ \pi(x)=\pi(x)\circ T \text{ for all } x\in \h \} \]
 is a real division associative algebra, hence 
 End$\,_\h (V_i)= \R, \C$ or $\H$, the real and complex
numbers or the quaternions, respectively. We say that $V_i$ is of real, complex or quaternionic type if 
 End$\,_\h (V_i)= \R, \C$ or $\H$, respectively. It follows that 
\begin{equation}\label{intertwining}
\text{End}\,_\h (V)=\g\l(r_1, \F_1) \oplus \cdots \oplus \g\l(r_k, \F_k),
\end{equation}
where $\F_i=\R, \C$ or $\H$ depending on the type of $V_i$. 

Let $\h$  be a compact Lie algebra,  $(\pi ,W)$ an irreducible real representation of $\h$ and consider the representation $W^{\oplus r}, \, r>0$ 
($r$ copies). Fix a $\pi(\h )$-invariant inner product $\pint$ on $W^{\oplus r}$. It follows that  $\text{End}\,_\h (W^{\oplus r}) =\g\l(r,\F)$ where $\F= \R, \C$ or $\H$. Therefore,
\begin{equation}\label{eq:cases} \text{End}\,_\h (W^{\oplus r})\cap O\left(W^{\oplus r} , \pint \right) =
\begin{cases} O(r), & \text{ if } \F= \R, \\
U(r)=GL(r, \C)\cap O(2r), & \text{ if } \F= \C, \\
Sp(r)=GL(r,\H)\cap U(2r), & \text{ if } \F= \H . 
\end{cases}
\end{equation} 


\begin{lemma}\label{lem:orth} 
 Let $\h$  be a compact Lie algebra,  $(\pi ,W)$ an irreducible real representation of $\h$ and consider the representation $W^{\oplus r}, \, r>0$ 
($r$ copies). Fix a $\pi(\h )$-invariant inner product $\pint$ on $W^{\oplus r}$ and assume that any of the following conditions holds: 
\begin{enumerate}
\item[$\ri$]  $W$ is of real type and $r$ is even, 
\item[$\rii$]  $W$ is of complex or quaternionic type.  
\end{enumerate}
Then there exists $J\in \operatorname{End}\,_\h (W^{\oplus r})\cap O\left(W^{\oplus r} , \pint \right) $ such that  $J^2=-I$.
\end{lemma}
\begin{proof}
Assume that (i) holds and let $r=2k$. Decompose $W^{\oplus r}=W^{\oplus k}\oplus W^{\oplus k}$ orthogonally and let $J(u,w)=(-w,u), \; u,w\in W^{\oplus k}$. Clearly, $J$ satisfies the conclusion of the statement.

If (ii) holds, then $\operatorname{End}\,_\h (W^{\oplus r}) \cap O\left(W^{\oplus r} , \pint \right)= U(r)$ or $Sp(r)$ (see \eqref{eq:cases}). Therefore, there exists 
$J\in \operatorname{End}\,_\h (W^{\oplus r}) \cap O\left(W^{\oplus r} , \pint \right)$ such that  $J ^2=-I$, and the lemma follows. 
\end{proof}

\begin{remark}
It was proved in \cite[Corollary 5.16]{BM} that a naturally reductive $2$-step nilpotent Lie group endowed with a left invariant metric does not admit orthogonal bi-invariant complex structures. In contrast with this result, the next proposition shows that there are plenty of naturally reductive $2$-step homogeneous Riemannian nilmanifols admitting orthogonal abelian complex structures.  
\end{remark}

\begin{proposition}\label{natred} Let $(N(\h, V), \pint)$ be any naturally reductive $2$-step homogeneous Riemannian nilmanifold  and let $s=0$ or $1$, where   $\dim \h \equiv s \pmod{2}$.    Then, the following conditions are equivalent:
\begin{itemize}
\item[$\ri$] 
 $(\R^s\oplus  N(\h, V), \pint)$ admits an orthogonal abelian complex structure,
\item[$\rii$]   
each irreducible subrepresentation of $V$ of real type has even multiplicity. 
\end{itemize}
\end{proposition}
\begin{proof} In case $s=1$, when $\dim \h$ is odd, the inner product on $N(\h, V)$ is extended to $\R \oplus  N(\h, V)$ so that $\R$ is orthogonal to $\n$ and the representation $(V,\pi)$ of $\h$ is extended to $\R\oplus  \g$ by $\pi(\R)=0$, so that the corresponding naturally reductive space has a 1-dimensional Euclidean factor.

 Let $J$ be an orthogonal abelian complex structure on $(\R^s\oplus  N(\h, V), \pint)$, then $J\z\subset \z$ and $JV\subset V$, where $\z=\R^s\oplus  \h$. The maps $j(z), \; z\in \z$, defined in \eqref{jz} are in this case given by $\pi(z)$ where $(V,\pi)$ is the real representation of $\R^s \oplus \h$.   It follows from  Lemma \ref{rem:abelian} (i)  that $J\in \text{End}\,_\h (V)=\g\l(r_1, \F_1) \oplus \cdots \oplus \g\l(r_k, \F_k)$ (see \eqref{intertwining}), where $V=V_1^{\oplus r_1}\oplus \cdots \oplus V_k^{\oplus r_k}  $ is the orthogonal decomposition of $V$ into isotypic components. Let $\pint _i$ be the restriction of $\pint$ to $V_i^{\oplus r_i}$.   The restriction $J_i$ of $J$ to $V_i^{\oplus r_i}$ satisfies $J_i\in \text{End}\,_\h (V_i^{\oplus r_i})\cap O\left(V_i^{\oplus ri} , \pint _i \right)$ and 
 $J_i^2=- I$, where $I$ is the identity on $V_i^{ \oplus r_i}$.  If $V_i$ is irreducible of real type, then $\text{End}\,_\h (V_i^{\oplus r_i})$ is isomorphic to   $\g\l(r_i, \R)$, so we would have  ${J_i}\in \g\l(r_i, \R)$ such that ${J_i}^2=- I$ where $I$ is the identity on $\R^{r_i}$.    Such an endomorphism  exists  if and only if $r_i$ is even.  

For the converse, assuming that  $r_i$ is even for every isotypic component $V_i^{\oplus r_i}$ of $V$ with $V_i$ irreducible of real type, an orthogonal abelian complex structure on $(\R^s\oplus  N(\h, V), \pint)$ is obtained as follows. Take an arbitrary orthogonal  endomorphism $J_\z$ of $\z=\R^s\oplus  \h$ such that $J_\z^2=-I$ where $I$ is the identity on $\z$. Lemma \ref{lem:orth} implies that, for each $i$, there exists $J_i\in \text{End}\,_\h (V_i^{r_i}) \cap O\left(V_i ^{\oplus r_i} , \pint _i \right)$,   such that ${J_i}^2=- I$. 
Define $J_V$ on $V$ by $J_V|_{V_i^{\oplus r_i}}=J_i, \; i=1, \dots , k$. It follows from Lemma \ref{rem:abelian}  (i) that $J\in \text{End}\,(\R^s\oplus  \n )$ defined by $J|_\z=J_\z$, $J|_V=J_V$ is an orthogonal abelian complex structure on $(\R^s\oplus  N(\h, V), \pint)$. 
\end{proof}
 The above proposition has its counterpart in the case of abelian hypercomplex structures. In fact, the  result analogous to Proposition \ref{natred} in the hypercomplex case is Proposition \ref{natred_hcx} below, whose proof makes use of the following result, similar to Lemma \ref{lem:orth}.   
\begin{lemma}\label{lem:orth2} Let $\h$  be a compact Lie algebra,  $(\pi ,W)$ an irreducible real representation of $\h$ and consider the representation $W^{\oplus r}, \, r>0$
($r$ copies). Fix a $\pi(\h )$-invariant inner product $\pint$ on $W^{\oplus r}$ and assume that any of the following conditions holds:  
\begin{enumerate}
\item[$\ri$]  $W$ is of real type and $r \equiv 0 \pmod{4}$,   
\item[$\rii$]  $W$ is of complex type and $r$ is even, 
\item[$\riii$]  $W$ is of quaternionic type.  
\end{enumerate}
Then there exist $J_\al \in \operatorname{End}\,_\h (W^{\oplus r})\, \cap\, O\left(W^{\oplus r},\pint \right),\, \al=1,2,3,$ such that \[
J^2_\al=-I, \quad \al=1,2,3,\qquad  J_1J_2=-J_2J_1=J_3.\]
\end{lemma}
\begin{proof}
If (i) holds, let $r=4k$ and decompose $W^{\oplus r}=W^{\oplus k}\oplus W^{\oplus k} \oplus W^{\oplus k}\oplus W^{\oplus k}$ orthogonally. For $x,y,z,w\in W^{\oplus k}$,  set
\[ J_1(x,y,z,w)=(-y,x,-w,z), \quad J_2(x,y,z,w)=(-z,w,x,-y), \quad J_3=J_1J_2.\]
It follows that these endomorphisms satisfy the required properties. 

Assume next that (ii)  holds, so  $r=2s$ and decompose $W^{\oplus r}=W^{\oplus s}\oplus W^{\oplus s}$ orthogonally. Since $W$ is of complex type, Lemma \ref{lem:orth} gives that there exists $J\in \operatorname{End}\,_\h (W^{\oplus s}) \cap O\left( W^{\oplus s}, \pint \right)$ such that  $J^2=-I$, where we still denote by  $\pint$  the restriction of $\pint$ to $W^{\oplus s}$. For $u,w \in W^{\oplus s}$, set 
\[ J_1(u,w)=(Ju,-Jw), \qquad  J_2(u,w)=(-w,u), \qquad J_3=J_1J_2.           \]
We have that $J_1,\, J_2$ and $J_3$  satisfy the conclusion of the statement. 

Finally, if (iii) holds, then  $\operatorname{End}\,_\h (W^{\oplus r})\cap O\left(W^{\oplus r} , \pint \right)=Sp(r)$ so there exist $J_\al \in \operatorname{End}\,_\h (W^{\oplus r})\cap O\left(W^{\oplus r} , \pint \right)$, $\al=1,2,3$, satisfying the required properties.

%
\end{proof}

We recall that a hypercomplex structure on a Lie algebra $\g$ is a triple $\hcx$, $\al=1,2,3$, of complex structures satisfying  $J_1J_2=-J_2J_1=J_3$. 
The hypercomplex structure is said to be abelian when each complex structure $J_\al$ is abelian, $\al=1,2,3$.
An inner product $\pint$ on  $(\g, \hcx )$ is called hyper-Hermitian when $\pint$ is Hermitian with respect to $J_\al$, $\al=1,2,3$. 
The study of abelian hypercomplex structures on Lie groups of Heisenberg type (see \cite{K}), was carried out  in  \cite{Bar}. 
\begin{remark}\label{rem:hkt}
The study of left invariant abelian hypercomplex structures on nilmanifolds  has an interesting application in hyper-K\"ahler with torsion (or HKT) geometry \cite{HP}. 
It was shown in \cite{DF} that on a nilmanifold endowed with such a  structure,  any left invariant hyper-Hermitian metric is automatically HKT. Indeed, it was  proved in  \cite{DF}  that   a hyper-Hermitian structure $(\hcx ,\pint )$ on $\g$ is HKT if and only if the following condition is satisfied:
\begin{align}
    \la [J_1x,J_1y],z\ra +\la [J_1y,J_1z],x\ra+ & \la [J_1z,J_1x],y\ra = \nonumber \\ 
 & = \la [J_2x,J_2y],z\ra +\la [J_2y,J_2z],x\ra +
\la J_2z,J_2x],y\ra \label{inv_hkt}\\
&  =\la [J_3x,J_3y],z\ra+\la [J_3y,J_3z],x\ra +\la J_3z,J_3x],y\ra \nonumber
\end{align}
for all  $x,y,z \in {\g}$. It is clear that \eqref{inv_hkt} holds when $\hcx$ is abelian.
\end{remark}
We give next necessary and sufficient conditions on a naturally reductive $2$-step homogeneous Riemannian nilmanifold to admit an abelian hypercomplex structure such that the corresponding Riemannian metric is hyper-Hermitian. As a consequence, we obtain a large family of naturally reductive metrics which are HKT.

\begin{proposition}\label{natred_hcx}
Let $(N(\h, V), \pint)$ be any naturally reductive $2$-step homogeneous Riemannian nilmanifold without Euclidean factor and  set $s=4-j$ if $\dim\h \equiv j \pmod 4$, $1\leq j \leq 4$.  Then, the following conditions are equivalent:
\begin{itemize}
\item[$\ri$] 
$(\R^s\oplus  N(\h, V), \pint)$ admits an  abelian hypercomplex structure such that $\pint$ is hyper-Hermitian,   
\item[$\rii$] each irreducible subrepresentation $V_i$ of $V$ of real type has  multiplicity $r_i\equiv 0 \pmod 4$ and  each irreducible subrepresentation  of $V$ of complex type has even  multiplicity. 
\end{itemize}
In particular, if any of these equivalent conditions is satisfied, then the naturally reductive inner product $\pint$ is HKT. 
\end{proposition}
\begin{proof} It follows from Lemma \ref{rem:abelian} (i) that $\hcx$ is an abelian hypercomplex structure on $(\R^s\oplus  N(\h, V), \pint)$ such that $\pint$ is hyper-Hermitian if and only if $\R^s\oplus\h$ and $V$ are $J_\al$-invariant and $J_\al |_V\in  \text{End}\,_\h (V)\cap O(V, \pint)=\left(\g\l(r_1, \F_1) \oplus \cdots \oplus \g\l(r_k, \F_k)\right) \cap O(V, \pint)$ for $\al=1,2,3$, where $V=V_1^{\oplus r_1}\oplus \cdots \oplus V_k^{\oplus r_k}  $ is the orthogonal decomposition of $V$ into isotypic components. The proof is analogous to that of Proposition \ref{natred} by making use 
of Lemma \ref{lem:orth2}. 

The last assertion follows from Remark \ref{rem:hkt}, since the hypercomplex structure is abelian.
\end{proof}

\

\ 


\begin{thebibliography}{99}\frenchspacing

\bibitem{ABD}
A. Andrada, M. L.  Barberis, I.  Dotti, 
Classification of abelian complex structures on 6-dimensional Lie algebras, 
\textit{ J. Lond. Math. Soc., II. Ser. } \textbf{83} (2011),  232--255.

\bibitem{AD}
A . Andrada, I. Dotti, Killing-Yano 2-forms on 2-step nilpotent Lie groups, \textit{ 
Geom. Dedicata} \textbf{212}  (2021), 415--424.

\bibitem{AV}
A. Andrada, R. Villacampa, 
Abelian balanced Hermitian structures on unimodular Lie algebras,  
\textit{Transform. Groups} \textbf{21} (2016), 903--927.

\bibitem{AN}
R.M. Arroyo, M. Nicolini, SKT structures on nilmanifolds, \textit{ Math. Z.} 
\textbf{
 302} (2022), 1307--1320.


\bibitem{Bar}
M. L. Barberis, Abelian hypercomplex structures on central extensions of H-type Lie algebras, \textit{ J. Pure  Appl. Algebra} 
\textbf{158} (1)   (2001),  15--23.

\bibitem{BD}
M. L. Barberis, I. Dotti, {\it Hypercomplex structures on a class of solvable Lie groups}, Quart. J.
Math. Oxford {\bf 47} (1996), 389--404.

\bibitem{BDV}
M. L. Barberis, I. Dotti, M. Verbitsky, Canonical bundles of complex nilmanifolds, with applications to hypercomplex geometry, \textit{
Math. Res. Lett.} \textbf{16} (2009),  331--347.



\bibitem{Bro}
R.W. Brockett, Control theory and singular Riemannian geometry, in \textit{New Directions in
Applied Mathematics}, P.J. Hilton and G.S. Young eds., pages 11--27, Springer, Berlin
(1981). 

\bibitem{CG}
G. R. Cavalcanti, M. Gualtieri, Generalized complex structures on nilmanifolds,
\textit{ J. Symplectic Geom.} \textbf{2} (3) (2004), 393--410.

\bibitem{COUV}
 M. Ceballos, A. Otal, L. Ugarte, R. Villacampa, Invariant complex structures on $6$-nilmanifolds: classification,
Fr\"olicher spectral sequence and special Hermitian metrics, \textit{J. Geom. Anal.} \textbf{26} (1) (2016),  252--286.

\bibitem{CFP}
S. Console, A. Fino,  Y. S. Poon, Stability of abelian complex structures, 
\textit{Internat. J. Math.} \textbf{17} (2006), 401--416.

\bibitem{CFGU1}
L. A. Cordero, M. Fern\'andez, A. Gray, L. Ugarte, The holomorphic fibre bundle structure of some compact complex
nilmanifolds, \textit{Proceedings of the 1st International Meeting on Geometry and Topology}, Braga, Portugal,  1997. Braga: Univ. do Minho, Centro de Matem\'atica, 207--221 (1997). 



\bibitem{CFGU}
 L. Cordero, M. Fern\'andez, A. Gray, L. Ugarte, Compact nilmanifolds with nilpotent complex structures: Dolbeault
cohomology, \textit{Trans. Amer. Math. Soc.} \textbf{352} (12) (2000),  5405--5433.


\bibitem{BM}
V. del Barco, A. Moroianu, 
Killing forms on $2$-step nilmanifolds, \textit{
J. Geom. Anal.} \textbf{31} (1) (2021), 863--887.

\bibitem{Bis}
J. -M. Bismut,  A local index theorem for non-Kähler manifolds,  \textit{Math. Ann.} {\bf 284} (1989), 681--699.

\bibitem{De}
J. Der\'e, Orthogonal bi-invariant complex structures on metric Lie
algebras, \textit{Ann. Global Anal. Geom. } \textbf{ 59} (2021), 157--177. 

\bibitem{DLV}
A.J. Di Scala, J. Lauret, L. Vezzoni,  Quasi-Kähler Chern-flat manifolds and complex 2-step nilpotent
Lie algebras, \textit{Ann. Sc. Norm. Super. Pisa Cl. Sci. (5)} \textbf{11} (2012), 41--60.

\bibitem{DF}
I. Dotti, A. Fino,  Hyper-K\"ahler with torsion
structures invariant by nilpotent Lie groups, \textit{Class. Quantum Grav.} \textbf{19} (2002), 1--12.

\bibitem{DF1}
I. Dotti, A. Fino, Hypercomplex eight-dimensional nilpotent Lie groups, \textit{J. Pure Appl. Algebra} \textbf{184} (2003), 41--57.

\bibitem{Eb1}
P. Eberlein,  Geometry of $2$-step nilpotent Lie groups with a left-invariant metric, \textit{Ann. Sci. Éc.
Norm. Sup. (4)} \textbf{27} (1994), 611--660.

\bibitem{Eb2}
P. Eberlein, Geometry of $2$-step nilpotent Lie groups with a left-invariant metric II, \textit{Trans.
Amer. Math. Soc.} \textbf{343} (1994), 805--828.

\bibitem{EFV}
N. Enrietti, A. Fino, L. Vezzoni,  Tamed symplectic forms and strong Kähler with torsion metrics, \textit{ J.
Symplect. Geom.} \textbf{ 10} (2012), 203--223.

\bibitem{FTV}
A. Fino, N. Tardini, L. Vezzoni, Pluriclosed and Strominger K\"ahler-like metrics compatible with abelian complex structures,  \textit{Bull. London Math. Soc.} \textbf{54 } (2022), 1862--1872.

\bibitem{Fol}
G.B. Folland, Subelliptic estimates and function spaces on nilpotent
Lie groups, \textit{Ark. Mat.} \textbf{13} (1975), 161--207.


\bibitem{GZZ}
Q. Gao,  Q. Zhao,  F. Zheng, Maximal nilpotent complex structures, \textit{  Transformation Groups} \textbf{28} (2023), 241--284. 

\bibitem{G} 
C. Gordon, Naturally reductive homogeneous riemannian manifolds, \textit{ Canad. J. Math.}
\textbf{37} (1985), 467--487. 

\bibitem{Good}
R. Goodman, Nilpotent Lie Groups, Structure and Applications to Analysis, \textit{Lecture Notes
in Math.} Vol. 562, Springer, Berlin (1976).

\bibitem{Gua}
Z.-D. Guan,  Toward a classification of compact nilmanifolds with
symplectic structures, \textit{Int. Math. Res. Not.}  \textbf{2010} (2010), 4377--4384.

\bibitem{HP}
P. S. Howe, G. Papadopoulos, Twistor spaces for hyper-K\"ahler manifolds with torsion, \textit{Phys. Lett. B} \textbf{379} (1996),  80--86.



\bibitem{K}
A. Kaplan,  On the geometry of groups of Heisenberg type, \textit{Bull. London Math. Soc.} 
\textbf{15} (1983), 35--42. 

\bibitem{KN}
S. Kobayashi, K. Nomizu, Foundations of differential geometry II, Interscience, New York, 1969.

\bibitem{LUV}
 A. Latorre, L. Ugarte, R. Villacampa, The ascending central series of nilpotent Lie algebras with complex structures,
\textit{Trans. Amer. Math. Soc.} \textbf{372}  (2019), 3867--3903.

\bibitem{LUV1}
 A. Latorre, L. Ugarte, R. Villacampa, Complex structures on nilpotent Lie algebras with one-dimensional center, \textit{J. Algebra} \textbf{614} (2023), 271--306.

\bibitem{Lau}
J. Lauret, Homogeneous nilmanifolds attached to
representations of compact Lie groups, \textit{Manuscripta Math.}  \textbf{99} (1999), 287--309. 

\bibitem{Mag}
L. Magnin, Sur les alg\`ebres de Lie nilpotentes de dimension $\leq 7$, \textit{J. Geom. Phys.} \textbf{3} (1986), 119--144.

\bibitem{Mor}
V. Morosov, Classification of nilpotent Lie algebras of order $6$, \textit{Izv. Vyssh. Uchebn. Zaved. Mat.} \textbf{4} 
(1958) 161--171.

\bibitem{OS}
G.P. Ovando, M. Subils, Symplectic structures on low dimensional 2-step nilmanifolds, \textit{Mat. Contemp.}\textbf{ 60} (2024), 225--254.
  

\bibitem{Pansu}
P. Pansu, Métriques de Carnot-Carathéodory et quasiisométries des espaces symétriques de rang un, \textit{
Ann. Math.} (2) \textbf{129} (1989), 1--60.

\bibitem{PT}
H. Pouseele, P. Tirao, 
Compact symplectic nilmanifolds associated with graphs,
\textit{J. Pure Appl. Algebra} \textbf{ 213} (2009),  1788--1794.

\bibitem{Rol}
S. Rollenske, Geometry of nilmanifolds with left-invariant complex structure and deformations in the large, \textit{Proc.
London Math. Soc.} \textbf{99} (2009),  425--460.
 
\bibitem{Sal}
 S. Salamon, Complex structures on nilpotent Lie algebras, \textit{J. Pure Appl. Algebra} \textbf{157} (2001), 311--333. 

\bibitem{Str}
A. Strominger, Superstrings with torsion, \textit{ Nucl. Phys. B} {\bf 274} (1986), 253--284.

\bibitem{Tam}
H. Tamaru, Two-step nilpotent Lie groups and homogeneous fiber bundles, 
\textit{Ann. Global Anal. Geom.} \textbf{24}  (2003), 53--66.

\bibitem{TV}
F. Tricerri, L. Vanhecke,  Naturally reductive homogeneous spaces and generalized
Heisenberg groups, \textit{Compositio Math.} \textbf{52} (1984), 389--408. 

\bibitem{Ug}
 L. Ugarte, Hermitian structures on six-dimensional nilmanifolds, \textit{Transform. Groups} \textbf{12}(1) (2007), 175--202.

\bibitem{Z} 
J. Zhang, 
Complex structures on stratified Lie algebras, \textit{
Bull. Aust. Math. Soc.} \textbf{106} (2022),  320--332.
\end{thebibliography}
\end{document}